\pdfoutput=1
\documentclass{wsbart/wsbart}

\usepackage[utf8]{inputenc}
\usepackage[T1]{fontenc}
\usepackage{lmodern}

\usepackage[maxnames=99,style=trad-plain,backend=biber,block=space,
uniquename=init]{biblatex}
\newcommand{\citelink}[2]{\hyperlink{cite.\therefsection @#1}{#2}}
\DeclareRedundantLanguages{English}{english}

\usepackage{titling}
\predate{}
\postdate{}

\usepackage{amsmath,amssymb}
\usepackage{amsthm}
\usepackage{bm,dsfont,stmaryrd}
\SetSymbolFont{stmry}{bold}{U}{stmry}{m}{n}
\renewcommand{\mathbb}{\mathds}
\usepackage{mathtools}
\usepackage{accents}
\usepackage{IEEEtrantools}

\newcommand{\R}{\mathbb R}
\newcommand{\T}{\mathbb T}
\newcommand{\D}{\mathcal D}
\newcommand{\E}{\mathcal E}
\newcommand{\vect}[1]{\bm{#1}}
\newcommand{\x}{\vect{x}}
\newcommand{\cco}{\llbracket}
\newcommand{\ccf}{\rrbracket}
\newcommand{\dd}{\mathop{}\!\mathrm{d}}
\DeclareMathOperator{\Law}{Law}

\DeclareMathOperator{\Expect}{\mathbb{E}}
\newcommand{\1}{\mathbb{1}}
\newcommand{\colonbin}{\mathbin{:}}

\newcommand{\CJW}{C_\textnormal{JW}}

\newtheorem{thm}{Theorem}
\newtheorem{lem}[thm]{Lemma}
\newtheorem{cor}[thm]{Corollary}
\newtheorem{prop}[thm]{Proposition}
\theoremstyle{definition}

\newtheorem*{assu*}{Main assumption}
\theoremstyle{remark}
\newtheorem{rem}{Remark}

\title{Sharp local propagation of chaos\\for mean field particles
with $W^{-1,\infty}$ kernels}
\author{Songbo Wang}
\date{}
\datefootnote{\today}

\begin{document}

\normalsizeAffilfont
\maketitle

\begin{abstract}
We study a system of $N$ diffusive particles with
$W^{-1,\infty}$ mean field interaction and establish
$O(1/N^2)$ local propagation of chaos estimates as
$N \to \infty$, measured in relative entropy
and in weighted $L^2$ distance. These results extend
the work of Lacker
[\citelink{LackerHier}{\textit{Probab.\ Math.\ Phys.},
4(2):377--432, 2023}] to singular interactions. The
entropy bound follows from a hierarchy of relative
entropies and Fisher informations, and applies to the
2D viscous vortex model in the weak interaction regime
regime, yielding a uniform-in-time estimate. The
$L^2$ bound is obtained through a hierarchy of
$\chi^2$ divergences and Dirichlet energies, leading
to sharp short-time estimates for the same model
without constraints on the interaction strength.
\end{abstract}

\section{Introduction and main results}
\label{sec:intro}

In this work, we are interested in the following system
of $N \geqslant 2$ interacting particles
on the $d$-dimensional torus $\T^d = (\R / \mathbb Z)^d$:
\begin{equation}
\label{eq:ps}
\dd X^i_t = \frac 1{N-1} \sum_{j \in [N] : j \neq i}
K\bigl(X^i_t - X^j_t\bigr) \dd t + \sqrt 2 \dd W^i_t,
\qquad\text{for $i \in [N]$},
\end{equation}
where $K$ is a singular interaction force kernel,
$W^i_\cdot$ are independent Brownian motions.
and $[N] \coloneqq \cco 1, N\ccf = \{1, \ldots, N\}$.
To be precise, we will consider force kernels admitting the decomposition
$K = K_1 + K_2$
such that $K_1$ is divergence-free and belongs to
the homogeneous space $\dot W^{-1,\infty} (\T^d; \R^d)$,
in the sense that
$K_{1,\alpha} = \sum_{\beta=1}^d \partial_{\beta} V_{\beta \alpha}$
for some matrix field $V \in L^\infty(\T^d ;\R^{d\times d})$,
and $K_2 \in L^\infty (\T^d; \R^d)$.
For simplicity we write $W^{-1,\infty} = \dot W^{-1,\infty}$ in the following.
We then write the particle system's formal mean field limit when $N \to \infty$:
\begin{equation}
\label{eq:mf}
\dd X_t = (K\star m_t) \dd t + \sqrt 2 \dd W_t,
\qquad m_t = \Law(X_t),
\end{equation}
and wish to show that the system \eqref{eq:ps} converges to \eqref{eq:mf}
when $N \to \infty$ in an appropriate sense.

The main example of the system in singular interaction
is the \emph{2D viscous vortex model},
where $d = 2$
and $K$ is a periodic version of the following force kernel defined on $\R^2$:
\[
K' (x) = \frac{1}{2\pi} \frac{x^\perp}{\lvert x\rvert^2}
= \frac{1}{2\pi} \biggl( - \frac{x_2}{\lvert x\rvert^2},
\frac{x_1}{\lvert x\rvert^2} \biggr)^{\!\top},\qquad
x = (x_1,x_2)^\top.
\]
Notice that we have $K' = \nabla \cdot V'$ for
\[
V' (x) = \frac{1}{2\pi}
\begin{pmatrix}
- \arctan (x_2/x_1) & 0 \\
0 & \arctan (x_1/x_2)
\end{pmatrix}.
\]
The model originates from the studies
of 2D incompressible Navier--Stokes equations
and we refer readers to the work of Jabin and Z.~Wang \cite{JabinWang}
and the expository article \cite{SaintRaymondVortex}
(and references therein) for details.

Throughout the paper, we suppose that
the $N$ particles in the dynamics \eqref{eq:ps} are exchangeable,
that is, for all permutation $\sigma$ of the index set $[N]$, we have
$\Law\bigl(X^1_t, \ldots, X^N_t\bigr)
= \Law\bigl(X^{\sigma(1)}_t, \ldots, X^{\sigma(N)}_t\bigr)$,
and denote $m^{N,k}_t = \Law\bigl(X^1_t,\ldots,X^k_t\bigr)$.
The aim of this paper is then to investigate quantitatively the convergence
$m^{N,k}_t \to m_t^{\otimes k}$
when $N \to \infty$ and $k$ remains fixed.
This corresponds to the quantitative
\emph{propagation of chaos} (PoC) phenomenon in the sense of Kac; see
\textcite{HaurayMischlerKac} for details.
To measure the difference between probability measures,
we use the relative entropy
\begin{align*}
H (m_1 | m_2)
&= \int \log \frac{m_1(x)}{m_2(x)} m_1(\dd x)
\intertext{and the $\chi^2$ divergence}
D (m_1 | m_2)
&= \int \biggl( \frac{m_1(x)}{m_2(x)} - 1 \biggr)^{\!2} m_2(\dd x)
\end{align*}
The relative entropy acts as a weighted $L\log L$ norm of the relative
density $m_1/m_2$, while the $\chi^2$ divergence corresponds to a
weighted $L^2$ norm of the same density. Moreover, the latter controls
the former via an interpolation-type argument. For convenience, we
sometimes call the $\chi^2$ divergence the $L^2$ distance, whenever
this causes no ambiguity.
In both of the two equations above, we have identified
the probability laws $m_1$, $m_2$
with their density functions
(with respect to the appropriate Lebesgue measure).
The results of this paper are thus upper bounds on
\[
H^k_t = H\bigl( m^{N,k}_t \big| m_t^{\otimes k} \bigr),~
D^k_t = D\bigl( m^{N,k}_t \big| m_t^{\otimes k} \bigr)
\]
that are diminishing when $N \to \infty$.
In the case of diffusion processes, the two crucial quantities
\begin{align*}
I (m_1 | m_2)
&= \int \biggl| \nabla \log \frac{m_1(x)}{m_2(x)} \biggr|^2 m_1(\dd x), \\
E (m_1 | m_2)
&= \int \biggl| \nabla \frac{m_1(x)}{m_2(x)} \biggr|^2 m_2(\dd x),
\end{align*}
called respectively (relative) Fisher information and Dirichlet energy,
also appear when we study the time-evolution of the relative entropy
and the $L^2$ distance.
In fact, the inclusion of these quantities in the analysis
is the main novelty of this work.

Recently, the propagation of chaos phenomenon of singular mean field dynamics
has raised high interests.
The main approach to handle singular interactions
is to construct suitable weak-strong stability functionals
that compare the $N$-particle and mean field marginal flows.
The $N$-particle marginal flow
\[t \mapsto m^{N}_t \coloneqq m^{N,N}_t \coloneqq
\Law\bigl( X^1_t, \ldots, X^N_t \bigr)\]
satisfies the \emph{Liouville},
\emph{Fokker--Planck} or \emph{forward Kolmogorov equation}
\begin{equation}
\label{eq:ps-fp}
\partial_t m^N_t
= \sum_{i \in [N]} \Delta_i m^N_t
- \frac 1{N-1} \sum_{i,j \in [N]: i\neq j}
\nabla_i \cdot \bigl( m^N_t K(x^i - x^j) \bigr).
\end{equation}
Notice that the $N$-tensorization $m_t^{\otimes N}$
of the mean field system \eqref{eq:mf} solves
\begin{equation}
\label{eq:mf-fp}
\partial_t m^{\otimes N}_t
= \sum_{i \in [N]} \Delta_i m^{\otimes N}_t
- \sum_{i \in [N]}
\nabla_i \cdot \bigl( m^{\otimes N}_t (K\star m_t)(x^i) \bigr).
\end{equation}
For $W^{-1,\infty}$ force kernels with bounded divergences
under possibly vanishing diffisivity,
Jabin and Z.~Wang \cite{JabinWang} showed that the relative entropy functional
suffices for the weak-strong stability,
yielding global PoC estimates that grows exponentially in time.%
\footnote{This work will be referred as ``Jabin--Wang''
in the following of this paper
without including the name initial of the second author.}
For deterministic dynamics
with repulsive or conservative Coulomb and Riesz interactions,
Serfaty constructed the modulated energy in \cite{SerfatyME}
and derived global-in-time PoC.
Then, Bresch, Jabin and Z.~Wang \cite{BJWMFE,BJWAttractive}
extended the method of Serfaty
to diffusive (and possibly attractive) Coulomb and Riesz systems
and showed the global-in-time PoC
by marrying relative entropy with modulated energy,
the new functional being called modulated free energy.
We mention here also another work \cite{CdCRSAttractive}
on the attractive case with logarithmic potentials.
More recently, by analyzing the decay of the mean-field limit and
exploiting dissipation through functional inequalities,
\textcite{GLBMVortex} and \textcite{CdCRSUniform} obtained
uniform-in-time PoC estimates for the 2D viscous vortex model and for
diffusive Coulomb flows, respectively. Extensions to the whole space
were carried out in \cite{FWVortex,lsiut,RSWhole1}.

The main result of \cite{JabinWang} applied to our dynamics
\eqref{eq:ps}, \eqref{eq:mf} already indicates
\begin{align*}
H \bigl( m^N_t \big| m_t^{\otimes N} \bigr) &\leqslant C e^{Ct}
\intertext{for some $C \geqslant 0$,
if the initial distance is zero: $m^N_0 = m_0^{\otimes N}$.
Then by the super-additivity of relative entropy, we get}
H \bigl( m^{N,k}_t \big| m_t^{\otimes k} \bigr)
&\leqslant \frac{C e^{Ct}}{\lfloor N/k \rfloor},
\end{align*}
and this is already a quantitative PoC estimate.
However, the findings of Lacker in \cite{LackerHier} reveal that
the $O(k/N)$-order bound obtained above is sub-optimal
for regular interactions (where $K$ is e.g.\ bounded),
and the sharp order in this case is $O(k^2\!/N^2)$.
The method of Lacker is to consider the BBGKY hierarchy
of the marginal distrbutions $(m^{N,k}_t)_{k \in [N]}$,
where the evolution of $m^{N,k}_t$ depends on itself
and the higher-level marginal $m^{N,k+1}_t$, namely
\begin{equation}
\label{eq:BBGKY}
\begin{aligned}
\partial_t m^{N,k}_t
&= \sum_{i\in [k]}\Delta_i m^{N,k}_t
- \frac 1{N-1}\sum_{i,j\in [k] : i\neq j}
\nabla_i \cdot \bigl(m^{N,k}_t K(x^i - x^j)\bigr) \\
&\mathrel{\hphantom{=}} \quad \negmedspace {}
- \frac {N-k}{N-1} \sum_{i \in [k]}
\nabla_i \cdot \biggl( \int_{\T^d} K(x^i - x_*) m^{N,k+1}_t
(\x^{[k]}, x^*) \dd x^* \biggr),
\end{aligned}
\end{equation}
and then to calculate the evolution
of $H^k_t = H \bigl( m_t^{N,k} \big| m_t^{\otimes k} \bigr)$,
which yields a hierarchy of ODE where
$\dd H^k_t /\! \dd t$ depends on $H^k_t$ and $H^{k+1}_t$.
Solving this ODE system allows for the sharp $O(k^2\!/N^2)$ bounds on $H^k_t$.
This method of Lacker is \emph{local} in the sense that
the quantity of interest describes the behavior of a fixed number of particles
even when $N \to \infty$, and stand in contrast with
the \emph{global} approaches mentioned in the paragraph above,
where the $N$-particle joint law is instead considered.
Then, in collaboration with Le Flem, Lacker \cite{LLFSharp} strengthened this
result by proving a uniform-in-time $O(k^2\!/N^2)$ rate in a weak interaction
regime, relying on log-Sobolev inequalities to exploit heat dissipation.
Very recently, \textcite{HCRHigher}
extended this hierarchical method to the $L^2$ distance
and obtained sharp convergence rates for higher-order expansions
in the case of bounded force kernels
(the convergence of $m^{N,k}_t$ to the tensorized law $m_t^{\otimes k}$
being merely zeroth-order).
\textcite{XieUniformTimeSizeChaos} considered the same framework
but adopted a different approach, obtaining uniform-in-time estimates for
cumulant functions of arbitrary order.

The entropy and $L^2$ methods require non-zero diffusivity in the dynamics
to yield sharp chaos estimates,
thus excluding deterministic Vlasov dynamics considered in the recent work
of Duerinckx \cite{DuerinckxChaos}.
Still, these methods enable two improvements.
First, the norm-distance between $m^{N,k}_t$ and $m_t^{\otimes k}$
(which scales as the square root of relative entropy)
can be shown to be of order $O(k/N)$,
while directly applying the correlation bounds in \cite{DuerinckxChaos}
gives only an $O(k^2\!/N)$-order control.
Note that this is also the order obtained in \cite{PPSChaos}
for dynamics with collision terms.
Second, the entropy and $L^2$ methods fully exploit the Laplace operator
to prevent the loss of derivatives in the BBGKY hierarchy
and establish chaos bounds in stronger norms
than those in \cite{DuerinckxChaos}.

Finally, we note that Bresch, Jabin and coauthors
have also applied hierarchical methods
to study second-order dynamics of singular interaction
in recent works \cite{BJSNewApproach,BDJDuality},
and have shown respectively short-time strong PoC and global-in-time weak PoC
under different regularity assumptions.
This is significant progress,
as the previous best PoC results for second-order systems,
to the knowledge of the author, apply only to mildly singular force kernels
satisfying $K(x) = O(\lvert x\rvert^{-\alpha})$ for $\alpha < 1$.
\textcite{BLBJSLongtime} further address the short-time limitation
in this method by exploiting the heat dissipation in first-order dynamics.

In this work,
we extend the entropic hierarchy of Lacker
and the $L^2$ hierarchy of Hess-Childs--Rowan (only in the zeroth-order)
to the case of $W^{-1,\infty}$ interactions.
In the new hierarchies of ODE,
which describe the evolution of $H^k_t$ and $D^k_t$ respectively,
Fisher information and Dirichlet energy of the next level appear,
and we develop new methods to solve the ODE systems.
In the first entropic case, we show that
$H^k_t = O(k^2\!/N^2)$ globally in time,
if the interaction strength is weak enough
(or equivalently, upon a rescaling of time, the interaction is weak enough).
Moreover, in the case of 2D vortex model,
we show that and $H^k_t = O(k^2e^{-rt}\!/N^2)$ for some $r > 0$,
thanks to the exponential decay established in \cite{GLBMVortex,CdCRSUniform}.
We also provide a simple way to solve Lacker's ODE system,
based on a comparison principle.
In the second $L^2$ case, we remove the restriction on the interaction strength
by working with $L^2$ distances $D^k_t$
and show that $D^k_t = O(1/N^2)$ for $k = O(1)$
but only in a short time interval.

\bigskip

We state the main results and discuss them in the rest of this section,
and give their proof in Section~\ref{sec:proof}.
The studies of the ODE hierarchies, which are the final steps of the proof
and the main technical contributions of this work,
are postponed to Section~\ref{sec:hier}.
We present some other technical results in Section~\ref{sec:toolbox}.

\subsection*{Main results}

Throughout the paper, we will work with a solution
of the Liouville equation \eqref{eq:ps-fp}, denoted $m^N_t$, for which
we can find a sequence of force kernels
$K^\varepsilon \in \mathcal C^\infty(\T^d)$
and probability densities $m^{N,\varepsilon}_t \in \mathcal C^\infty(\T^d)$
such that they satisfy \eqref{eq:ps-fp} when
$K$, $m^N_t$ are respectively replaced
by $K^\varepsilon$, $m^{N,\varepsilon}_t$;
that $K^\varepsilon \to K$ almost everywhere
and $m^{N,\varepsilon}_t \to m^N_t$ weakly as probability measures;
and finally that $m^{N,\varepsilon}_t$ is bounded below by a positive constant.
We suppose also that the mean field flow $m_t$
is the weak limit of $\mathcal C^\infty$ approximations $m^{\varepsilon}_t$
that correspond to the McKean--Vlasov SDE \eqref{eq:mf} driven by
the regularized force kernel $K^\varepsilon$,
and that each $m^\varepsilon_t$ has also strictly positive density.
In particular, the 2D viscous vortex model verifies this assumption.
See e.g.\ \cite{lsiut} for details.
(Although the setting there is on $\R^d$ instead of
$\T^d$ but the argument is the same.)
We impose this technical assumption in order to avoid subtle
well-posedness issues in the singular PDE \eqref{eq:ps-fp}.
\textcite{JabinWang} considers entropy solutions, but it is not clear
to the author whether this notion is equivalent
to the regularized one adopted here.

We present the main assumption of this paper concerning the regularity
of the force kernel.

\begin{assu*}
The interaction force kernel admits the decomposition
$K = K_1 + K_2$, where $K_1 = \nabla \cdot V$
for some $V \in L^\infty(\T^d; \R^d \times \R^d)$
and satisfies $\nabla \cdot K_1 = 0$, and $K_2 \in L^\infty$.
\end{assu*}

We then state our main results.

\begin{thm}[Entropic PoC]
\label{thm:entropy}
Let the main assumption hold.
Suppose that the marginal relative entropies at the initial time satisfy
\[
H^k_0 \leqslant C_0 \frac{k^2}{N^2}
\]
for all $k \in [N]$, for some $C_0 \geqslant 0$.
If $\lVert V\rVert_{L^{\infty}} < 1$,
then for all $T > 0$, there exists $M$,
depending on
\[
C_0,~\lVert V\rVert_{L^{\infty}},~
\lVert K_2\rVert_{L^\infty},~
\sup_{t \in [0,T]} \lVert \nabla \log m_t\rVert_{L^\infty}^2
\!+ \lVert \nabla^2 \log m_t\rVert_{L^\infty},
\]
such that for all $t \in [0,T]$,
\[
H^k_t \leqslant M e^{Mt}\frac{k^2}{N^2}.
\]
If additionally $K_2 = 0$ and
\[
\lVert \nabla \log m_t \rVert_{L^\infty}^2
+ \lVert \nabla^2 \log m_t \rVert_{L^\infty}
\leqslant M_m e^{ - \eta t}
\]
for all $t \geqslant 0$, for some $M_m \geqslant 0$ and $\eta > 0$, then
for all $r$ such that
$0 < r < r_*
\coloneqq \min\bigl( \eta, (1 - \lVert V\rVert_{L^\infty}) 8 \pi^2 \bigr)$,
there exists $M'$, depending on
\[
C_0,~\lVert V\rVert_{L^\infty},~M_m,~\eta,~r,~d,
\]
such that for all $t \geqslant 0$, we have
\[
H^k_t \leqslant M' e^{- rt} \frac{k^2}{N^2}.
\]
\end{thm}

\begin{rem}
In the case $K_2 = 0$ of Theorem~\ref{thm:entropy},
choosing $m_t$ as the uniform distribution on $\T^d$,
yields an exponential rate of local convergence of the particle system
towards its stationary state.
\end{rem}

\begin{thm}[$L^2$ PoC]
\label{thm:l2}
Let the main assumption hold.
Suppose that the marginal $L^2$ distances at the initial time satisfy
\[
\sum_{k=1}^N r^k D^k_0 \leqslant \frac{C_0}{N^2(1-r)^3}
\]
for all $k \in [N]$ and $r \in [0,1)$, for some $C_0 \geqslant 0$.
Let $T > 0$ be arbitrary.
If the matrix field $V$ satisfies
\[
M_V \coloneqq \sup_{t\in [0,T]}
\sup_{x \in \T^d}\int_{\T^d} \lvert V(x-y)\rvert^2 m_t(\dd y) < 1,
\]
then there exists $T_* > 0$,
depending on
\[
\lVert V\rVert_{L^{\infty}},~M_V,~
\lVert K_2\rVert_{L^\infty},~
\sup_{t \in [0,T]} \lVert \nabla \log m_t\rVert_{L^\infty}^2\!+
\lVert \nabla^2 \log m_t\rVert_{L^\infty},
\]
such that for all $t \in [0,T_* \wedge T)$, we have
\[
D^k_t \leqslant \frac{Me^{Mk}}{(T_*-t)^3N^2}.
\]
for some $M$ depending additionally on $C_0$.
\end{thm}

\begin{rem}
The $e^{Mk}$ dependency on $k$ in Theorem~\ref{thm:l2}
is highly suboptimal and appears to be a proof artifact,
at least for $k = o(N)$.
We do not pursue this refinement since it does not resolve
the more significant short-time limitation.
\end{rem}

\begin{rem}
The $O(k^2\!/N^2)$
entropic estimate in Theorem~\ref{thm:entropy} appears sharp.
Indeed, set
\[
(\delta H)^k_t = H^{k+1}_t - H^k_t,\quad
(\delta^2 H)^k_t = (\delta H)^{k+1}_t - (\delta H)^k_t.
\]
By the entropy chain rule,
\[
(\delta^2 H)^k_t =
\Expect \bigl[ H \bigl( \Law(X^{k+1}_t, X^{k+2}_t | X^1_t,\ldots,X^k_t)
\big| \Law(X^{k+1}_t | X^1_t,\ldots,X^k_t)^{\otimes 2} \bigr) \bigr].
\]
In other words, each $(\delta^2 H)^k_t$ corresponds to a conditional
$2$-cumulant, which is expected to be $O(1/N)$ for mean field interactions.
Since relative entropy scales quadratically in the small scale,
we expect each $(\delta^2 H)^k_t$ to be $O(1/N^2)$.
By the difference relation
\[
H^k_t
= k H^1_t + \sum_{\ell=1}^{k-1} \sum_{n=0}^{\ell-1}
(\delta^2 H)^n_t,
\]
we expect $H^k_t$ to be $O(k^2\!/N^2)$.
Gaussian dynamics on the whole space saturates this bound
\cite[Example~2.8]{LackerHier},
although an explicit example on the torus is,
to the author's knowledge, still unknown.

As noted in the introduction, the $\chi^2$ divergence dominates the
relative entropy. Thus, Theorem~\ref{thm:l2} appears sharp as $N \to \infty$
with $k$ fixed. However, the precise $\chi^2$ behavior when both $k$ and
$N \to \infty$ remains unknown, even for regular interactions.
\end{rem}

\subsection*{Further remarks}

\subsubsection*{\boldmath$\nabla\cdot K_1 = 0$ is not restrictive}
First, as noted in \cite{JabinWang},
the condition that the singular part $K_1$ is divergence-free
is not restrictive.
Indeed, if the interaction force kernel $K$ admits the decomposition
$K = K'_1 + K'_2$, where both $K'_1$ and $\nabla \cdot K'_1$ belong to
$W^{-1,\infty}$ (which is the regularity assumption of \cite{JabinWang}),
and $K'_2 \in L^\infty$,
we can find, by definition, a bounded vector field $S$ such that
$\nabla \cdot K'_1 = \nabla \cdot S$.
By shifting the components of $S$ by constants,
we can also suppose without loss of generality that
this vector field verifies $\int_{\T^d} S = 0$.
Thus, we have the alternative decomposition
\[
K = (K'_1 - S) + (K'_2 + S),
\]
where the first part $K'_1 - S$ is divergence-free
and the second part $K'_2 + S$ is bounded.
Since $S \in L^\infty$ and $\int_{\T^d} S=0$, we can find a bounded matrix field
$V_S$ such that $\nabla \cdot V_S = S$
and $\lVert V_S\rVert_{L^\infty} \leqslant C_d \lVert S\rVert_{L^\infty}$
for some $C_d$ depending only on the dimension $d$.%
\footnote{For example one can take
$V^{1i}_S(x^1,x^2,\ldots,x^d) = \int_0^{x^1} S^i(y,x^2,\ldots,x^d) \dd y$
for $i \in [d]$ and $V_S^{ji} = 0$ for $j \neq 1$.}
So the new decomposition satisfies the main assumption
and it only remains to verify the respective ``smallness'' conditions
of the two theorems for the force kernel $K'_1-S$.

\subsubsection*{Weak vortex interaction}
Second, Theorem~\ref{thm:entropy} applies to the 2D viscous vortex model
if the vortex interaction is weak enough.
Indeed, in the vortex case, we have $K = \nabla \cdot V$
for some $V \in L^\infty$ and $\nabla \cdot K = 0$
so the main assumption is satisfied with $K_2 = 0$.
The required regularity bounds for the mean field flow $m_t$
have been established in \cite{GLBMVortex,CdCRSUniform}.
More precisely, it is shown in \cite[Section 3.2]{CdCRSUniform} that
if the initial value $m_0$ of the mean field equation belongs to
$W^{2,\infty}(\T^d)$ and verifies the lower bound $\inf m_0 > 0$,
then we have the required decaying bound on the regularity:
\[
\lVert \nabla \log m_t\rVert_{L^\infty}^2
+ \lVert \nabla^2 \log m_t \rVert_{L^\infty}
\leqslant M_m e^{-\eta t}.%
\footnote{The rate of convergence stated in \cite{CdCRSUniform}
is not explicit. However, it seems to the author that
we can take $\eta = 4\pi^2$ by the following argument.
First by computing the evolution of the entropy $H(m_t)$
and integrating by parts à la Jabin--Wang, we find that
$\dd H(m_t)/\!\dd t = - I(m_t) \leqslant - 8\pi^2 H(m_t)$
thanks to the log-Sobolev inequality
(see also \cite[Proof of Theorem 4.11]{lsiut}),
and therefore $H(m_t) \lesssim e^{-8\pi^2t}$.
This implies that $\lVert m_t - 1\rVert_{L^1} \lesssim e^{-4\pi^2t}$
by Pinsker.
Then we use the hypercontractivity
\cite[Corollary 2.4]{CdCRSUniform}
and the regularization \cite[Proposition 2.6]{CdCRSUniform} to find that
$\lVert \nabla m_t\rVert_{L^\infty}$,
$\lVert \nabla^2m_t\rVert_{L^\infty} \lesssim e^{-4\pi^2 t}$
so the desired bound follows with $\eta = 4\pi^2$.
This rate is optimal as it is verified by the heat equation ($K=0$)
with initial data $m_0(x) = 1 + a \sin (2\pi x) + b \cos(2 \pi x)$.
With $\eta = 4\pi^2$, the minimum for the rate in the second assertion
of Theorem~\ref{thm:entropy} is equal to
$\min(1,2 - 2\lVert V\rVert_{L^\infty})4\pi^2$.}
\]
So Theorem~\ref{thm:entropy} applies if $\lVert V\rVert_{L^\infty} < 1$.
By scaling arguments, this is equivalent to a high viscosity
or high temperature condition.
In this regime,
the second assertion of Theorem~\ref{thm:entropy}
provides a finer long-time convergence estimate
on the relative entropies for the 2D viscous vortex model
compared to the global results in \cite{GLBMVortex,CdCRSUniform},
which apply more generally without this weak interaction restriction.
It is unclear to the author
if the weak interaction restriction can be lifted;
see also the discussion on $L^2$ results in below.

\subsubsection*{\boldmath$L^d$ interaction of any strength}
On the contrary, if the interaction force kernel $K$
is of the slightly higher regularity class
\[
K \in L^d,~\nabla \cdot K \in L^d,
\]
then Theorem~\ref{thm:entropy} can be applied without any restriction
on the strengh of $K$.
To this end, we consider
$K^\varepsilon = K \star \rho^\varepsilon$ where $\rho^\varepsilon$
is a sequence of $\mathcal C^\infty$ mollifiers on $\T^d$.
Since $\int_{\T^d} K - K^\varepsilon = 0$ and
$\int_{\T^d} \nabla \cdot K - \nabla \cdot K^\varepsilon = 0$,
the result of Bourgain and Brezis \cite{BourgainBrezis} indicates that
we can find a matrix field $V$ and a vector field $S$ on $\T^d$
solving the equations $\nabla \cdot V = K - K^\varepsilon$
and $\nabla \cdot S = \nabla \cdot K - \nabla \cdot K^\varepsilon$
with the bounds
\begin{align*}
\lVert V\rVert_{L^\infty} &\leqslant C_d
\lVert K - K^\varepsilon\rVert_{L^\infty}, \\
\lVert S\rVert_{L^\infty} &\leqslant C_d
\lVert \nabla \cdot K - \nabla \cdot K^\varepsilon\rVert_{L^\infty}
\end{align*}
for some $C_d > 0$ depending only on $d$.
By shifting the components of $S$, we can suppose that
$\int_{\T^d} S = 0$ and this does not alter the $L^\infty$ bound on $S$ above.
We find again a matrix field $V_S$ such that $\nabla \cdot V_S = S$
and $\lVert V_S\rVert_{L^\infty} \leqslant C_d \lVert S\rVert_{L^\infty}$.
Then we decompose the force kernel $K$ in the following way:
\[
K = (K - K^\varepsilon) + K^\varepsilon
= \nabla \cdot V + K^\varepsilon
= \nabla \cdot (V - V_S) + (K^\varepsilon + S).
\]
By construction, the singular part is divergence-free:
\[
\nabla^2\colonbin(V-V_S) = \nabla \cdot (K - K^\varepsilon)
- \nabla \cdot S = 0,
\]
and the remaining part $K^\varepsilon + S$ is bounded, so the main assumption
is satisfied.
The $W^{-1,\infty}$ norm of the singular part is controlled by
\[
\lVert V - V_S \rVert_{L^\infty}
\leqslant \lVert V\rVert_{L^\infty} + \lVert V_S\rVert_{L^\infty}
\leqslant C_d \bigl( \lVert K - K^\varepsilon\rVert_{L^d}
+ \lVert \nabla \cdot K - \nabla \cdot K^\varepsilon\rVert_{L^d} \bigr).
\]
Yet, the mollification is continuous in $L^d$:
\[
\lVert K - K^\varepsilon\rVert_{L^d},~
\lVert \nabla \cdot K - \nabla \cdot K^\varepsilon\rVert_{L^d}
\to 0,\qquad\text{when $\varepsilon\to 0$.}
\]
So in order to apply Theorem~\ref{thm:entropy}, it suffices to take
an $\varepsilon$ small enough.
In a previous work, Han \cite[Theorem 1.2]{HanEnt} derived
global $O(1/N^2)$ PoC under the assumption that
$K$ is divergence-free and belongs to $L^p$ for some $p > d$,
and the $N$-particle initial measure satisfies
the density bound $\lambda^{-1} \leqslant m^N_0 \leqslant \lambda$
uniformly in $N$.
In comparison to this work,
our method achieves two major improvements: first,
the critical Krylov--Röckner exponent $p = d$ is treated \cite{KrylovRockner};
and second, the rather demanding condition on $m^N_0$
(which excludes non-trivial chaotic data $m^N_0 = m_0^{\otimes N}$
for $m_0 \neq 1$) is lifted.
These improvements are made possible by our consideration
of the new hierarchy involving Fisher information
(see Proposition~\ref{prop:ent-hier})
and a Jabin--Wang type large deviation estimate
(see Corollary~\ref{cor:concentration}).

\subsubsection*{Vortex interaction of any strength by \boldmath$L^2$}
By a similar regularity trick,
the $L^2$ result of Theorem~\ref{thm:l2}
can be applied to the 2D viscous vortex model of any interaction strength.
Indeed, as in the case, $K = \nabla \cdot V$ for $V \in L^\infty$
and $\nabla \cdot K = 0$, we can decompose
\[
K = (K - K^\varepsilon) + K^\varepsilon
= \nabla \cdot (V - V^\varepsilon) + K^\varepsilon,
\]
where $K^\varepsilon = K\star \rho^\varepsilon$ and
$V^\varepsilon = V\star \rho^\varepsilon$.
Then the $L^2$ constant in Theorem~\ref{thm:l2} satisfies
\[
M_{V - V^\varepsilon} \coloneqq \sup_{t\in [0,T]}
\sup_{x \in \T^d}\int_{\T^d} \lvert (V-V^\varepsilon)(x-y)\rvert^2 m_t(\dd y)
\leqslant \lVert V - V^\varepsilon \rVert_{L^2}^2
\sup_{t \in [0,T]} \lVert m_t \rVert_{L^\infty},
\]
and can be arbitrarily small as $\varepsilon \to 0$.
Thus Theorem~\ref{thm:l2} gives an $O(1/N^2)$ PoC estimate in short time.
Since our treatment of the $L^2$ hierarchy in Proposition~\ref{prop:l2-hier}
is rather crude, it seems possible to the author that
the explosion in finite time is sub-optimal.
Here, the major technical difficulty is that
we cannot force the hierarchy to stop at a certain level
$k \sim N^\alpha$, $\alpha < 1$ as done in Hess-Child--Rowan \cite{HCRHigher}.
And this is due to the fact that
we do not have a priori bounds on $L^2$ distances and Dirichlet energies
that are strong enough.

\subsubsection*{Dynamics on the whole space}
The global-in-time framework of Theorem~\ref{thm:entropy}
extends to the whole space.
In the proof of the theorem,
the only obstruction is that
$\nabla \log m_t$ is no longer in $L^\infty$.
Yet \textcite{FWVortex} recently proved that on the whole space,
\[
\lvert \nabla \log m_t (x) \rvert \leqslant C e^{Ct} (1 + \lvert x\rvert).
\]
This regularity bound is sufficient for the Jabin--Wang method,
which controls the inner interaction terms in the proof.
Proposition~\ref{prop:transport} can likewise be modified
to handle linear growth using the weighted Pinsker inequality
of \textcite{BolleyVillaniCKP},
which in turn controls the outer interaction terms.
For uniform-in-time estimates, one should add quadratic confinement
for the vortices and instead consider relative densities
with respect to a Gaussian;
see \cite{lsiut,RSWhole1} for details.

\section{Proof of Theorems~\ref{thm:entropy} and \ref{thm:l2}}
\label{sec:proof}

\subsection{Setup and proof outline}

In the proof we will work with regularized solutions
introduced in Section~\ref{sec:intro}
and prove the bounds in both theorems for these approximations.
Then the result holds for the original solutions
by lower semi-continuity. See \cite{lsiut} for details.

In the following,
we will perform the entropic and $L^2$ computations at the same time
in order to exploit the similarity between them.
We set $p = 1$ for the entropic computations
and $p = 2$ for the $L^2$ computations.
Then, we can write the relative entropy and the $L^2$ distance between
$m^{N,k}_t$ and $m_t^{\otimes k}$ formally as
\[
\D_p^k \coloneqq
\D_p \bigl( m^{N,k}_t \big| m_t^{\otimes k} \bigr) \coloneqq \frac{1}{p-1}
\biggl(\int_{\T^{kd}} \bigl(h^{N,k}_t\bigr)^p \dd m^{\otimes k}_t - 1\biggr),
\quad\text{where}~h^{N,k}_t \coloneqq \frac{m^{N,k}_t}{m_t^{\otimes k}}.
\]
The expression makes sense classically in the $L^2$ case where $p = 2$.
In the entropic case, this notation is motivated by the fact that
\[
\lim_{p \searrow 1}\frac{1}{p-1} \biggl( \int h^p \dd m - 1\biggr)
= \int h \log h \dd m
\]
for all postive $h$ that is upper and lower bounded (away from zero)
and all probability measure $m$ such that $\int h \dd m = 1$.

Then, we use the BBGKY hierarchy \eqref{eq:BBGKY}
and the tensorized mean field equation \eqref{eq:mf-fp}
to calculate the time derivative of $\D_p^k$.
We find
\begin{align*}
\frac{1}{p}\frac{\dd \D^k_p}{\dd t}
&= - \int_{\T^{kd}} \bigl( h^{N,k}_t\bigr)^{p-2}
\bigl| \nabla h^{N,k}_t \bigr|^2 \dd m_t^{\otimes k} \\
&\mathrel{\hphantom{=}}\negmedspace{} +
\begin{aligned}[t]\frac 1{N-1}\sum_{i,j\in [k] : i \neq j}
\int_{\T^{kd}}
&\bigl( h^{N,k}_t \bigr)^{p-1} \nabla_i h^{N,k}_t \\
&\quad\negmedspace{}
\cdot \bigl( K(x^i - x^j) - K \star m_t(x^i) \bigr)
m^{\otimes k}_t (\dd\x^{[k]}) \end{aligned}\\
&\mathrel{\hphantom{=}}\negmedspace{} +
\begin{aligned}[t] \frac{N-k}{N-1} \sum_{i\in [k]} \int_{\T^{kd}}
&\bigl( h^{N,k}_t \bigr)^{p-1} \nabla_i h^{N,k}_t \\
&\quad\negmedspace{}\mathrel{\cdot}
\Bigl< K(x^i - \cdot), m^{N,(k+1)|k}_{t} (\cdot | \x^{[k]}) - m_t \Bigr>
m_t^{\otimes k} (\dd\x^{[k]}),
\end{aligned}
\end{align*}
where the conditional measure $m^{N,(k+1)|k}_t (\cdot | \cdot)$ is defined as
\[
m^{N,(k+1)|k}_{t} (x^* | \x^{[k]})
\coloneqq \frac{m^{N,k+1}_t(\x^{[k]}, x^*)}{m^{N,k}_t(\x^{[k]})}
\]
Define also
\[
\E^k_p \coloneqq \int_{\T^{kd}} \bigl( h^{N,k}_t\bigr)^{p-2}
\bigl| \nabla h^{N,k}_t \bigr|^2 \dd m_t^{\otimes k}.
\]
This expression makes sense for both $p = 1$ and $2$,
and is the relative Fisher information
$I^k_t = I\bigl( m^{N,k}_t \big| m_t^{\otimes k} \bigr)$ for $p = 1$,
and the Dirichlet energy
$E^k_t = E\bigl( m^{N,k}_t \big| m_t^{\otimes k} \bigr)$ for $p = 2$.
Denote by $A$ and $B$
the last two terms in the equality above for $p^{-1}\dd D^k_p/\!\dd t$.
We find that $A = A_1 + A_2$ and $B = B_1 + B_2$ where
\[
A_a \coloneqq \frac 1{N-1}\sum_{i,j\in [k] : i \neq j}
\int_{\T^{dk}} \bigl( h^{N,k}_t \bigr)^{p-1} \nabla_i h^{N,k}_t
\cdot \bigl( K_a(x^i - x^j) - K_a \star m_t(x^i) \bigr)
m^{\otimes k}_t (\dd\x^{[k]})
\]
and
\begin{multline*}
B_a \coloneqq  \frac{N-k}{N-1} \sum_{i\in [k]}
\int_{\T^{dk}} \bigl( h^{N,k}_t \bigr)^{p-1} \nabla_i h^{N,k}_t \\
\cdot \Bigl< K_a(x^i - \cdot), m^{N,(k+1)|k}_{t} (\cdot | \x^{[k]}) - m_t \Bigr>
m_t^{\otimes k} (\dd\x^{[k]}),
\end{multline*}
for $a = 1$, $2$, since the expressions are linear in $K$
and the force kernel admits the decomposition $K = K_1 + K_2$.
Thus, the evolution of $\D^k_p$ writes
\[
\frac{1}{p}\frac{\dd \D^k_p}{\dd t}
= - \E^k_p + A_1 + A_2 + B_1 + B_2.
\]
We call $A_1$, $A_2$ the \emph{inner interaction} terms,
and $B_1$, $B_2$ the \emph{outer interaction} terms,
as the first two terms correspond to the interaction between the first $k$
particles themselves, and the last two terms to the interaction
between the first $k$ and the remaining $N - k$ particles.

We aim to find appropriate upper bounds for the last four interaction terms
$A_1$, $A_2$, $B_1$, $B_2$ in the rest of the proof.
To be precise, we will show in the entropic case $p=1$ the following
system of differential inequalities:
\[
\frac{\dd H^k_t}{\dd t}
\leqslant -c_1 I^k_t + c_2 I^{k+1}_t \1_{k < N}
+ M_1 H^k_t +  M_2 k \bigl( H^{k+1}_t - H^k_t \bigr) \1_{k < N}
+ M_3 \frac{k^\beta}{N^2},
\]
where $\beta$ is an integer $\geqslant 2$ and $c_1$, $c_2$,
$M_i$, $i \in [3]$ are nonnegative constants such that $c_1 > c_2$.
This hierarchy differs from that of \textcite{LackerHier}, as an additional term
$c_2 I^{k+1}_t$ is introduced to control the outer interaction terms,
reflecting the singularity of the force kernel.
This is due to the singularity of the force kernel.
In the $L^2$ case $p=2$, we show that
\[
\frac{\dd D^k_t}{\dd t}
\leqslant -c_1 E^k_t + c_2 E^{k+1}_t \1_{k < N}
+ M_2 k D^{k+1}_t \1_{k < N}
+ M_3 \frac{k^2}{N^2},
\]
where again $c_1 > c_2 \geqslant 0$ and $M_2$, $M_3 \geqslant 0$.
Again, the difference from \textcite{HCRHigher} lies in the inclusion
of the term $c_2 E^{k+1}_t$, required by the kernel singularity.
We will then apply the results from the following section
(Propositions~\ref{prop:ent-hier} and \ref{prop:l2-hier}) to solve
the hierarchies and this will conclude the proof.

\subsection{Two lemmas on inner interaction terms}

We present two lemmas that will be useful for controlling
the inner interactions terms $A_1$, $A_2$.
Their proofs are provided after their statements.
The first lemma treats two cases, $p = 1$ and $p = 2$.
The case $p = 1$ was established in \cite{LackerHier},
while the case $p = 2$ appears implicitly in \cite{HCRHigher}.
For completeness, we provide a full statement and proof here.
The second lemma, in the case $p = 1$, extends
\cite[Theorem~4]{JabinWang}.
Proofs are given after the statements.

\begin{lem}
\label{lem:internal-interaction-bounded}
Let $p \in \{1, 2\}$ and $k$ be an integer $\geqslant 2$.
Let $m \in \mathcal P(\T^d)$ and $h : \T^{kd} \to \R_{\geqslant 0}$
be exchangeable.
Suppose additionally that $\int_{\T^{kd}} h \dd m^{\otimes k} = 1$.
Let $U : \T^{2d} \to \R^d$ be bounded.
For $i \in [k]$, denote
\[
a \coloneqq \sum_{j\in [k]: j \neq i} \int_{\T^{kd}}
h^{p-1} \nabla_i h \cdot \bigl( U(x^i, x^j)
- \langle U(x^i,\cdot), m\rangle \bigr)
m^{\otimes k}(\dd\x^{[k]}),
\]
where $\langle U(x^i, \cdot), m\rangle = \int_{\T^d} U(x^i, y) m(\dd y)$.
Then in the case $p = 1$, we have for all $\varepsilon > 0$,%
\footnote{Here, and in the following,
if a bracket without conditions appears in a math expression,
it means that both alternatives are valid.}
\[
a \leqslant \varepsilon
\int_{\T^{kd}} \frac{\lvert \nabla_i h\rvert^2}{h} \dd m^{\otimes k}
+ \frac{\lVert U\rVert_{L^\infty}^2}{\varepsilon} \times \begin{cases}
(k-1)^2 \\
(k-1) + (k-1)(k-2)\sqrt{2H(m^3 | m^{\otimes 3})}
\end{cases}
\]
where $m^3$ is the $3$-marginal of the probability measure $h m^{\otimes k}$:
\[
m^3(\dd x^{[3]}) = \int_{\T^{(k-3)d}} h m^{\otimes k} \dd\x^{[k]\setminus[3]}.
\]
And in the case $p = 2$, we have for all $\varepsilon > 0$,
\[
a \leqslant \varepsilon
\int_{\T^{kd}} \lvert \nabla_i h\rvert^2 \dd m^{\otimes k}
+ \frac{2(k-1)^2\lVert U\rVert_{L^\infty}^2}{\varepsilon} D
+ \frac{2(k-1)\lVert U\rVert_{L^\infty}^2}{\varepsilon},
\]
where $D = \int_{\T^{kd}} (h-1)^2 \dd m^{\otimes k}$.
\end{lem}

\begin{lem}
\label{lem:internal-interaction-bounded-ibp}
Under the same setting as in Lemma~\ref{lem:internal-interaction-bounded},
let $\phi : \T^{2d} \to \R$ be a bounded function verifying
$\phi(x, x) = 0$ for all $x \in \T^d$ and
\[
\int_{\T^d} \phi(x,y) m(\dd y) = \int_{\T^d} \phi(y,x) m(\dd x) = 0,
\qquad\text{for all $x \in \T^d$.}
\]
Then we have
\begin{align*}
\MoveEqLeft
\sum_{i,j\in [k]} \int_{\T^{kd}}
h^p \phi(x^i, x^j) m^{\otimes k}(\dd\x^{[k]}) \\
&\leqslant \lVert \phi\rVert_{L^\infty}
\biggl[\sqrt{2\CJW} N
\biggl( \D_p + \frac{3k^2}{N^2} \biggr)
+ k^2 \D_p \1_{p=2} \biggr],
\end{align*}
where $\CJW$ is a universal constant to be defined in
Section~\ref{sec:jw-lem} and $\D_p$ is defined by
\[
\D_p \coloneqq \begin{cases}
\int_{\T^{kd}} h \log h \dd m^{\otimes k} & \text{when~$p = 1$,} \\
\int_{\T^{kd}} (h-1)^2 \dd m^{\otimes k} & \text{when~$p = 2$.}
\end{cases}
\]
\end{lem}

\begin{proof}[Proof of Lemma~\ref{lem:internal-interaction-bounded}]
In the simpler case $p = 2$, using the Cauchy--Schwarz inequality
\[
h \nabla_i h \cdot \xi
= \bigl( (h-1) + 1\bigr) \nabla_i h \cdot \xi
\leqslant \varepsilon \lvert \nabla_i h\rvert^2
+ \frac{1}{2\varepsilon} \bigl( (h-1)^2 + 1 \bigr) \lvert\xi\rvert^2,
\]
we get
\begin{multline*}
\sum_{j \in [k] : j \neq i}
h \nabla_i h \cdot \bigl( U(x^i, x^j) - \langle U(x^i,\cdot), m\rangle\bigr) \\
\leqslant \varepsilon\lvert\nabla_i h\rvert^2
+ \frac{1}{2\varepsilon}\bigl((h-1)^2 + 1\bigr)
\biggl| \sum_{j \in [k]: j \neq i}
\bigl(U(x^i, x^j) - \langle U(x^i, \cdot), m\rangle\bigr) \biggr|^2
\end{multline*}
Thus, integrating against $m^{\otimes k}$, we get
\begin{align*}
\MoveEqLeft\sum_{j\in [k]: j \neq i} \int_{\T^{kd}}
h^{p-1} \nabla_i h \cdot \bigl( U(x^i, x^j)
- \langle U(x^i, \cdot), m\rangle \bigr)
m^{\otimes k}(\dd\x^{[k]}) \\
&\leqslant \varepsilon
\int_{\T^{kd}} \lvert\nabla_i h\rvert^2\dd m^{\otimes k}\\
&\mathrel{\hphantom{=}}\quad\negmedspace{}
+ \frac{1}{2\varepsilon}\int_{\T^{kd}} \bigl((h-1)^2 + 1\bigr)
\biggl| \sum_{j \in [k]: j \neq i}
\bigl(U(x^i, x^j) - \langle U(x^i,\cdot),m\rangle \bigr) \biggr|^2
m^{\otimes k}(\dd\x^{[k]}) \\
&\leqslant \varepsilon \int_{\T^{kd}} \lvert\nabla_i h\rvert^2\dd\x^{[k]}
+ \frac{(k-1)^2\lVert U\rVert_{L^\infty}^2}{2\varepsilon} D\\
&\mathrel{\hphantom{=}}\quad\negmedspace{}
+ \frac{1}{2\varepsilon} \int_{\T^{kd}} \biggl| \sum_{j \in [k]: j \neq i}
\bigl(U(x^i, x^j) - \langle U(x^i,\cdot),m\rangle \biggr|^2
m^{\otimes k}(\dd\x^{[k]}).
\end{align*}
The integral in the last term is equal to
\[
\sum_{j_1, j_2\in [k] \setminus \{i\}}\int_{\T^{kd}}
\bigl(U(x^i, x^{j_1}) - \langle U(x^i,\cdot), m\rangle \bigr)
\cdot \bigl(U(x^i, x^{j_2}) - \langle U(x^i, \cdot), m\rangle\bigr)
m^{\otimes k}(\dd\x^{[k]}),
\]
and we notice that by independence,
the integral above does not vanish only if $j_1 = j_2$.
Thus we get the upper bound
\[
\int_{\T^{kd}} \biggl| \sum_{j \in [k]: j \neq i}
\bigl(U(x^i, x^j) - \langle U(x^i,\cdot), m\rangle(x^i) \biggr|^2
m^{\otimes k}(\dd\x^{[k]})
\leqslant 4(k-1)\lVert U\rVert_{L^\infty}^2,
\]
and this finishes the proof for the $p = 2$ case.

Now treat the entropic case where $p = 1$.
Using Cauchy--Schwarz, we get
\begin{multline*}
\sum_{j \in [k] : j \neq i}
\nabla_i h \cdot \bigl( U(x^i, x^j) - \langle U(x^i,\cdot), m\rangle \bigr) \\
\leqslant\varepsilon h^{-1}\lvert\nabla_i h\rvert^2
+ \frac{1}{4\varepsilon}
\biggl| \sum_{j \in [k]: j \neq i}
\bigl( U(x^i, x^j) - \langle U(x^i,\cdot), m\rangle \bigr) \biggr|^2.
\end{multline*}
Then integrating against $m^{\otimes k}$, we find
\begin{align*}
\MoveEqLeft \sum_{i,j \in [k] : j \neq i} \int_{\T^{kd}}
\nabla_i h \cdot \bigl( U(x^i, x^j) - \langle U(x^i, \cdot), m\rangle \bigr)
m^{\otimes k}(\dd\x^{[k]}) \\
&\leqslant \varepsilon
\int_{\T^{kd}} \frac{\lvert\nabla_i h\rvert^2}{h} \dd m^{\otimes k} \\
&\mathrel{\hphantom{=}}\quad\negmedspace {}
+ \frac{1}{4\varepsilon}\int_{\T^{kd}}
\biggl| \sum_{j \in [k]: j \neq i}
\bigl( U(x^i, x^j) - \langle U(x^i,\cdot), m\rangle \bigr) \biggr|^2
hm^{\otimes k}(\dd\x^{[k]}).
\end{align*}
So it remains to upper bound the last integral.
Employing the crude bound
\[
\biggl| \sum_{j \in [k]: j \neq i}
\bigl(U(x^i, x^j) - \langle U(x^i,\cdot), m\rangle\bigr) \biggr|^2
\leqslant 4(k-1)^2 \lVert U\rVert_{L^\infty}^2
\]
and the fact that $h m^{\otimes k}$ is a probability measure, we get
\[
\int_{\T^{kd}}
\biggl| \sum_{j \in [k]: j \neq i}
\bigl( U(x^i, x^j) - \langle U(x^i,\cdot), m\rangle \bigr) \biggr|^2
h m^{\otimes k}(\dd\x^{[k]})
\leqslant 4(k-1)^2 \lVert U\rVert_{L^\infty}^2.
\]
This yields the first claim for the case $p = 1$.
For the finer bound, we again expand the square in the integrand:
\begin{align*}
\MoveEqLeft\int_{\T^{kd}} \biggl| \sum_{j \in [k]: j \neq i}
\bigl( U(x^i, x^j) - \langle U(x^i,\cdot), m\rangle \bigr) \biggr|^2
h m^{\otimes k}(\dd\x^{[k]}) \\
&= \sum_{j \in [k] \setminus \{i\}}
\int_{\T^{kd}} \lvert U(x^i, x^j) - \langle U(x^i, \cdot), m\rangle\rvert^2
h m^{\otimes k} (\dd\x^{[k]}) \\
&\mathrel{\hphantom{=}} \quad \negmedspace {} + \begin{aligned}[t]
\sum_{j_1, j_2 \in [k] \setminus \{i\} : j_1 \neq j_2} \int_{\T^{kd}}
&\bigl( U(x^i, x^{j_1}) - \langle U(x^i,\cdot), m\rangle \bigr)\\
&\quad\cdot \bigl( U(x^i, x^{j_2}) - \langle U(x^i,\cdot),m\rangle \bigr)
h m^{\otimes k} (\dd\x^{[k]}).\end{aligned}
\end{align*}
The first term can be bounded crudely by
$4(k-1)\lVert U\rVert_{L^\infty}^2$ as before.
For the second term, we notice that the integration
against the measure $h m^{\otimes k}$ can be replaced by the integration
of the variables $x^i$, $x^{j_1}$, $x^{j_2}$ against the $3$-marginal
\[
m^3(\dd x^i \dd x^{j_1} \dd x^{j_2})
= \int_{\T^{(k-3)d}} h(\x^{[k]}) m^{\otimes k}(\x^{[k]})
\dd\x^{[k]\setminus\{i,j_1,j_2\}}.
\]
Notice that, by independence, we have
\[
\int_{\T^{3d}}
\bigl( U(x^i, x^{j_1}) - \langle U(x^i,\cdot), m\rangle \bigr)
\cdot \bigl( U(x^i, x^{j_2}) - \langle U(x^i,\cdot),m\rangle \bigr)
m^{\otimes 3}(\dd x^i \dd x^{j_1} \dd x^{j_2})
= 0.
\]
Using the Pinsker inequality between $m^3$ and $m^{\otimes 3}$, we find
for $j_1 \neq j_2$,
\begin{multline*}
\int_{\T^{3d}}
\bigl( U(x^i, x^{j_1}) - \langle U(x^i,\cdot), m\rangle \bigr)
\cdot \bigl( U(x^i, x^{j_2}) - \langle U(x^i,\cdot),m\rangle \bigr)
m^3(\dd x^i \dd x^{j_1} \dd x^{j_2}) \\
\leqslant 4 \lVert U \rVert_{L^\infty}^2
\sqrt{2 H(m^3 | m^{\otimes 3})},
\end{multline*}
and this concludes the proof for the case $p = 1$.
\end{proof}

\begin{proof}[Proof of Lemma~\ref{lem:internal-interaction-bounded-ibp}]
In the case $p = 1$, thanks to the Donsker--Varadhan duality, we have
\begin{align*}
\MoveEqLeft
\sum_{i,j \in [k]} \int_{\T^{kd}} h \phi(x^i,x^j)
m^{\otimes k} (\dd\x^{[k]}) \\
&= \sum_{i,j \in [k]} \int_{\T^{kd}} (h - 1) \phi(x^i,x^j)
m^{\otimes k} (\dd\x^{[k]}) \\
&\leqslant \eta^{-1} \int_{\T^{kd}} h \log h \dd m^{\otimes k}
+ \eta^{-1} \log \int_{\T^{kd}} \exp \biggl(
\eta \sum_{i,j\in [k]} \phi(x^i, x^j) \biggr)m^{\otimes k}(\dd\x^{[k]}),
\end{align*}
for all $\eta > 0$.
Then taking $\eta$ such that
$\sqrt{2\CJW}\lVert\phi\rVert_{L^\infty} N \eta = 1$
and applying the modified Jabin--Wang estimates
in Corollary~\ref{cor:concentration}, we get
\[
\sum_{i,j \in [k]} \int_{\T^{kd}} h \phi(x^i,x^j)
m^{\otimes k} (\dd\x^{[k]})
\leqslant
\sqrt{2\CJW} \lVert\phi\rVert_{L^\infty} N
\biggl( \D_1 + \frac{3k^2}{N^2} \biggr).
\]
In the case $p = 2$, we use the elementary equality
\[
h^2 = (h - 1)^2 + 2(h - 1) + 1
\]
and get
\begin{align*}
\MoveEqLeft
\sum_{i,j\in [k]} \int_{\T^{kd}}
h^2 \phi(x^i, x^j) m^{\otimes k}(\dd\x^{[k]}) \\
&= \sum_{i,j\in [k]} \int_{\T^{kd}}
(h-1)^2 \phi(x^i, x^j) m^{\otimes k}(\dd\x^{[k]}) \\
&\mathrel{\hphantom{=}}\negmedspace\quad{}
+ 2\sum_{i,j\in [k]} \int_{\T^{kd}}
(h-1) \phi(x^i, x^j) m^{\otimes k}(\dd\x^{[k]}) \\
&\leqslant k^2 \lVert\phi\rVert_{L^\infty}
\int_{\T^{kd}} (h-1)^2 \dd m^{\otimes k} \\
&\mathrel{\hphantom{=}}\negmedspace\quad{}
+ 2 \biggl( \int_{\T^{kd}} (h-1)^2 \dd m^{\otimes k}\biggr)^{\!1/2}
\biggl[ \int_{\T^{kd}}
\biggl( \sum_{i,j\in [k]} \phi(x^i, x^j) \biggr)^{\!2}
\dd m^{\otimes k} \biggr]^{1/2}
\end{align*}
The last integral has already been estimated
in the intermediate (and in fact the easiest) step
of the Jabin--Wang large deviation lemma
(see Proposition~\ref{prop:counting}):
\[
\int_{\T^{kd}}
\biggl(\sum_{i,j\in [k]} \phi(x^i, x^j) \biggr)^{\!2}
\dd m^{\otimes k}
\leqslant 2 k^2 \CJW \lVert \phi\rVert_{L^\infty}^2.
\]
Thus we have
\[
\int_{\T^{kd}} h^2 \phi(x^i,x^j) m^{\otimes k}(\dd\x^{[k]})
\leqslant k^2 \lVert \phi\rVert_{L^\infty} \D_2
+ 2 k \lVert \phi\rVert_{L^\infty} \sqrt{2 \CJW \D_2},
\]
so the desired result follows from the Cauchy--Schwarz inequality.
\end{proof}

\subsection{Control of the inner interaction terms}

In this step, we aim to find appropriate upper bounds
for the inner interactions terms
\[
A_a \coloneqq \frac 1{N-1}\sum_{i,j\in [k] : i \neq j}
\int_{\T^{dk}} \bigl( h^{N,k}_t \bigr)^{p-1} \nabla_i h^{N,k}_t
\cdot \bigl( K_a(x^i - x^j) - K_a \star m_t(x^i) \bigr)
m^{\otimes k}_t (\dd\x^{[k]}),
\]
where $p = 1$, $2$ and $a = 1$, $2$.

\subsubsection{Control of the regular part \boldmath$A_2$}

First start with the regular part.
In this case, we directly invoke Lemma~\ref{lem:internal-interaction-bounded}
with $U(x,y) = K_2(x - y)$
and $\varepsilon = (N-1)\varepsilon_1$
for some $\varepsilon_1>0$.
Summing over $i \in [k]$, we get
\begin{align*}
A_2 &\leqslant \varepsilon_1 I^k_t
+ \frac{C\lVert K_2\rVert_{L^\infty}^2k }{\varepsilon_1(N-1)^2}
\times \begin{cases}
(k-1)^2 \\
(k-1) + (k-1)(k-2) \sqrt{H^3_t}
\end{cases}
\intertext{for the case $p = 1$, and}
A_2 &\leqslant \varepsilon_1 E^k_t
+ \frac{C\lVert K_2\rVert_{L^\infty}^2k(k-1)^2}{\varepsilon_1(N-1)^2} D^k_t
+ \frac{C\lVert K_2\rVert_{L^\infty}^2k(k-1)}{\varepsilon_1(N-1)^2}
\end{align*}
for the case $p = 2$.
In both inequalities above,
$C$ denotes a universal constant that may change from line to line,
and we adopt this convention in the rest of the proof.

\subsubsection{Control of the singular part \boldmath$A_1$}
Recall that $K_1 = \nabla \cdot V$ and $\nabla \cdot K_1 = 0$.
Then we perform the integrations by parts:
\begin{align*}
\MoveEqLeft p(N-1)A_1 \\
&= p \sum_{i,j\in [k]:i\neq j} \int_{\T^{kd}}
\bigl( h^{N,k}_t \bigr)^{p-1} \nabla_i h^{N,k}_t
\cdot \bigl ( K_1(x^i - x^j) - (K_1\star m_t) (x^i) \bigr)
m_t^{\otimes k}(\dd\x^{[k]}) \\
&= \sum_{i,j\in [k]:i \neq j} \int_{\T^{kd}}
\nabla_i \bigl( h^{N,k}_t \bigr)^{p}
\cdot \bigl ( K_1(x^i - x^j) - (K_1\star m_t) (x^i) \bigr)
m_t^{\otimes k}(\dd\x^{[k]}) \\
&= \begin{aligned}[t]- \sum_{i,j\in [k]:i\neq j}
\int_{\T^{kd}} \bigl( h^{N,k}_t \bigr)^p
&\nabla \log m_t(x^i) \\
&\quad\cdot \bigl ( K_1(x^i - x^j) - (K_1\star m_t) (x^i) \bigr)
m_t^{\otimes k}(\dd\x^{[k]}) \end{aligned} \\
&= \begin{aligned}[t] \sum_{i,j\in[k]:i\neq j} \int_{\T^{kd}}
&\nabla_i \Bigl( \bigl(h^{N,k}_t\bigr)^p \nabla \log m_t(x^i) m_t^{\otimes k}
\Bigr) \\
&\qquad\colonbin \bigl( V(x^i-x^j) - (V\star m_t)(x^i)\bigr) \dd\x^{[k]}.
\end{aligned}
\end{align*}
Noticing that $\nabla \log m_t(x^i) m_t^{\otimes k} =
\nabla_i \bigl( m_t^{\otimes k} \bigr)$, we get
\begin{multline*}
\nabla_i \Bigl( \bigl(h^{N,k}_t\bigr)^p \nabla \log m_t(x^i) m_t^{\otimes k}
\Bigr) \\
= p \bigl( h^{N,k}_t\bigr)^{p-1} \nabla_i h^{N,k}_t \otimes
\nabla \log m_t(x^i) m_t^{\otimes k}
+ \bigl( h^{N,k}_t \bigr)^p \frac{\nabla^2 m_t(x^i)}{m_t(x^i)} m_t^{\otimes k}.
\end{multline*}
Hence,
\begin{align*}
\MoveEqLeft p(N-1)A_1 \\
&= \begin{aligned}[t] p \sum_{i,j\in[k]:i\neq j}
\int_{\T^{kd}} \bigl(h^{N,k}_t\bigr)^{p-1}
&\nabla_i h^{N,k}_t \otimes \nabla \log m_t(x^i) \\ &\quad \colonbin
\bigl( V(x^i-x^j) - (V\star m_t)(x^i)\bigr) m_t^{\otimes k}(\dd\x^{[k]})
\end{aligned}\\
&\mathrel{\hphantom{=}} \negmedspace{} + \sum_{i,j\in[k]:i\neq j}
\int_{\T^{kd}} \bigl(h^{N,k}_t\bigr)^p \frac{\nabla^2 m_t(x^i)}{m_t(x^i)}
\colonbin
\bigl( V(x^i-x^j) - (V\star m_t)(x^i)\bigr) m_t^{\otimes k}(\dd\x^{[k]}) \\
&\eqqcolon p(N-1) ( A_{11} + A_{12} ).
\end{align*}

For the first part $A_{11}$,
we invoke Lemma~\ref{lem:internal-interaction-bounded}
with $U(x,y) = \nabla \log m_t(x) \cdot V(x - y)$
and $\varepsilon = (N-1)\varepsilon_2$ for some $\varepsilon_2 > 0$.
Summing over $i \in [k]$, we get
\[
A_{11} \leqslant \varepsilon_2 I^k_t
+ \frac{C \lVert \nabla \log m_t\rVert_{L^\infty}^2
\lVert V\rVert_{L^\infty}^2 k}{\varepsilon_2(N-1)^2}
\times \begin{cases}
(k-1)^2 \\
(k-1) + (k-1)(k-2) \sqrt{H^3_t}
\end{cases}
\]
for the case $p=1$, and
\[
A_{11} \leqslant \varepsilon_2 E^k_t
+ \frac{C\lVert \nabla \log m_t\rVert_{L^\infty}^2
\lVert V\rVert_{L^\infty}^2k(k-1)^2}{\varepsilon_2(N-1)^2} D^k_t
+ \frac{C\lVert \nabla \log m_t\rVert_{L^\infty}^2
\lVert V\rVert_{L^\infty}^2k(k-1)}{\varepsilon_2(N-1)^2}
\]
for the case $p=2$.

For the second part $A_{12}$,
we invoke Lemma~\ref{lem:internal-interaction-bounded-ibp} with
\[
\phi(x,y) = \begin{cases}
\frac{\nabla^2 m_t(x)}{m_t(x)}
\colonbin \bigl( V(x - y) - (V\star m_t)(x) \bigr)
& \text{if $x \neq y$,} \\
0 & \text{if $x = y$.}
\end{cases}
\]
Note that the condition
\[
\int_{\T^d} \phi(x,y) m_t (\dd y) = \int_{\T^d} \phi(y,x) m_t(\dd y) = 0
\]
is verified due to the definition of convolution and the fact that
$\nabla^2 \colonbin V = \nabla \cdot K_1 = 0$.
Thus, we get
\[
A_{12} \leqslant \frac{\lVert \nabla^2 m_t / m_t\rVert_{L^\infty}
\lVert V\rVert_{L^\infty}}{N-1}
\biggl[ CN \biggl( \D^k_p + \frac{k^2}{N^2}\biggr) + k^2 \D^k_p \1_{p=2}\biggr]
\]
where $C$ is a universal constant.

Denote
\[
M_{V,m_t} \coloneqq
\lVert \nabla \log m_t\rVert_{L^\infty}^2 \lVert V\rVert_{L^\infty}^2
+ \lVert \nabla^2 m_t/m_t\rVert_{L^\infty}
\lVert V\rVert_{L^\infty},
\]
and note that here,
since $\nabla^2 m_t/m_t = (\nabla \log m_t)^{\otimes 2} + \nabla^2 \log m_t$,
the constant $M_{V,m_t}$ is finite by the assumptions of the theorems.
Summing up $A_{11}$ and $A_{12}$, we get
\begin{align*}
A_{1} &\leqslant \varepsilon_2 I^k_t
+ C M_{V,m_t}\biggl(H^k_t + \frac{k^2}{N^2} \biggr)
+ \frac{C M_{V,m_t} k}{\varepsilon_2N^2} \times \begin{cases}
k^2 \\
k + k^2 \sqrt {H^3_t}\end{cases}
\intertext{for the case $p=1$, and}
A_{1} &\leqslant \varepsilon_2 E^k_t
+ C M_{V,m_t}
\biggl( 1+ \frac{k^2}{N} + \frac{k^3}{\varepsilon_2N^2}\biggr) D^k_t
+ C M_{V,m_t} ( 1 + \varepsilon_2^{-1}) \frac{k^2}{N^2}
\end{align*}
for the case $p=2$.

\subsection{Control of the outer interaction terms}

Now we move on to the upper bounds for the terms $B_1$, $B_2$.
Recall that they are defined by
\begin{multline*}
B_a \coloneqq  \frac{N-k}{N-1} \sum_{i\in [k]}
\int_{\T^{dk}} \bigl( h^{N,k}_t \bigr)^{p-1} \nabla_i h^{N,k}_t \\
\cdot \Bigl< K_a(x^i - \cdot), m^{N,(k+1)|k}_{t} (\cdot | \x^{[k]}) - m_t \Bigr>
m_t^{\otimes k} (\dd\x^{[k]}),
\end{multline*}
where $p = 1$, $2$ and $a = 1$, $2$.

\subsubsection{Control of the regular part \boldmath$B_2$}

For the term $B_2$, we notice that in the entropic case, we have
by the Pinsker inequality
\begin{align*}
\Bigl|\Bigl< K_2(x^i - \cdot), m^{N,(k+1)|k}_{t} (\cdot | \x^{[k]}) - m_t \Bigr>
\Bigr|
&\leqslant \lVert K_2 \rVert_{L^\infty}
\sqrt{ 2 H \bigl( m^{N,(k+1)|k}_t(\cdot|\x^{[k]}) \big| m_t \bigr)},
\intertext{and in the $L^2$ case, we have}
\Bigl|\Bigl< K_2(x^i - \cdot), m^{N,(k+1)|k}_{t} (\cdot | \x^{[k]}) - m_t \Bigr>
\Bigr|
&\leqslant \lVert K_2 \rVert_{L^\infty}
\sqrt {D \bigl( m^{N,(k+1)|k}_t(\cdot|\x^{[k]}) \big| m_t \bigr)}.
\end{align*}
In both cases, we apply the Cauchy--Schwarz inequality
\begin{align*}
\MoveEqLeft \bigl( h^{N,k}_t \bigr)^{p-1} \nabla_i h^{N,k}_t
\cdot \Bigl< K_a(x^i - \cdot), m^{N,(k+1)|k}_{t} (\cdot | \x^{[k]}) - m_t \Bigr>
\\ &\leqslant
\frac{\varepsilon_3(N-1)}{N-k}
\bigl( h^{N,k}_t \bigr)^{p-2} \bigl|\nabla_i h^{N,k}_t\bigr|^2 \\
& \mathrel{\hphantom{\leqslant}} \quad \negmedspace {}
+ \frac{(N-k)}{4\varepsilon_3(N-1)}
\Bigl|\Bigl< K_2(x^i - \cdot), m^{N,(k+1)|k}_{t} (\cdot | \x^{[k]}) - m_t \Bigr>
\Bigr|^2.
\end{align*}
Integrating against the measure $m_t^{\otimes k}$ and summing over
$i \in [k]$, we get
\begin{align*}
B_2 &\leqslant \varepsilon_3 \E^k_p
+ \frac{\lVert K_2\rVert_{L^\infty}^2(N-k)^2k}{4\varepsilon_3(N-1)^2} \\
&\mathrel{\hphantom{\leqslant}} \hphantom{\varepsilon_3 \E^k_p}
\mathrel{\hphantom{+}} \quad \times \begin{cases}
\int_{\T^{kd}}
2 H \bigl( m^{N,(k+1)|k}_t(\cdot|\x^{[k]}) \big| m_t \bigr)
m_t^{\otimes k} (\dd\x^{[k]}) & \text{when~$p=1$}\\
\int_{\T^{kd}}
D\bigl( m^{N,(k+1)|k}_t(\cdot|\x^{[k]}) \big| m_t \bigr)
m_t^{\otimes k} (\dd\x^{[k]}) & \text{when~$p=2$}
\end{cases} \\
&= \varepsilon_3 \E^k_p
+ \frac{\lVert K_2\rVert_{L^\infty}^2(N-k)^2k}{2p\varepsilon_3(N-1)^2}
\bigl( \D^{k+1}_p - \D^k_p \bigr).
\end{align*}
The last equality is a ``towering'' property
of relative entropy and $\chi^2$ distance,
which can be verified directly from the definition of conditional density.

\subsubsection{Control of the singular part \boldmath$B_1$}
Applying the Cauchy--Schwarz inequality as in the previous step yields
\[
B_1 \leqslant \varepsilon_4 \E^k_p
+ \begin{aligned}[t] &\frac{(N-k)^2k}{4\varepsilon_4(N-1)^2} \\
&\quad\times\int_{\T^{kd}} \bigl( h^{N,k}_t\bigr)^p
\Bigl|\Bigl< K_1(x^i - \cdot), m^{N,(k+1)|k}_{t}(\cdot|\x^{[k]})
- m_t\Bigr>\Bigr|^2 m_t^{\otimes k} (\dd\x^{[k]}). \end{aligned}
\]

In the entropic case where $p=1$,
applying the first inequality of Proposition~\ref{prop:transport}
in Section~\ref{sec:toolbox} with
$m_1 \to m^{N,(k+1)|k}_{t}(\cdot|\x^{[k]})$,
$m_2 \to m_t$, we get
\begin{align*}
\MoveEqLeft
\Bigl|\Bigl< K_1(x^i - \cdot), m^{N,(k+1)|k}_{t}(\cdot|\x^{[k]})
- m_t\Bigr>\Bigr|^2 \\
&\leqslant \lVert V\rVert_{L^\infty}^2
(1+\varepsilon_5)I\bigl(m^{N,(k+1)|k}_{t}(\cdot|\x^{[k]})\big|m_t\bigr) \\
&\mathrel{\hphantom{\leqslant}} \quad \negmedspace {} +
2\lVert V\rVert_{L^\infty}^2
(1+\varepsilon_5^{-1})\lVert\nabla\log m_t\rVert_{L^\infty}^2
H\bigl(m^{N,(k+1)|k}_{t}(\cdot|\x^{[k]})\big|m_t\bigr).
\end{align*}
Noticing that the conditional entropy and Fisher information satisfy
the towering property:
\begin{align*}
\int_{\T^{kd}} H\bigl(m^{N,(k+1)|k}_t(\cdot|{\x^{[k]}})
\big|m_t\bigr) m^{N,k}_t(\dd\x^{[k]})
&= H^{k+1}_t - H^k_t, \\
\int_{\T^{kd}} I\bigl(m^{N,(k+1)|k}_t(\cdot|{\x^{[k]}})
\big|m_t\bigr) m^{N,k}_t(\dd\x^{[k]})
&= \frac{I^{k+1}_t}{k+1},
\end{align*}
we integrate the equality above with respect to $m^{N,k}_t$ and obtain
\begin{multline*}
B_1 \leqslant
\varepsilon_4 I^k_t
+ \frac{(1+\varepsilon_5)\lVert V\rVert_{L^\infty}^2(N-k)^2k}
{4\varepsilon_4(N-1)^2(k+1)} I^{k+1}_t \\
+ \frac{(1+\varepsilon_5^{-1})
\lVert V\rVert_{L^\infty}^2 \lVert\nabla\log m_t\rVert_{L^\infty}^2
(N-k)^2k}
{2\varepsilon_4(N-1)^2} \bigl(H^{k+1}_t - H^k_t\bigr).
\end{multline*}

In the $L^2$ case where $p=2$,
we apply the second inequality of Proposition~\ref{prop:transport}
in Section~\ref{sec:toolbox} with
$m_1 \to m^{N,(k+1)|k}_{t}(\cdot|\x^{[k]})$,
$m_2 \to m_t$, and get
\begin{align*}
\MoveEqLeft
\Bigl|\Bigl< K_2(x^i - \cdot), m^{N,(k+1)|k}_{t}(\cdot|\x^{[k]})
- m_t\Bigr>\Bigr|^2 \\
&\leqslant M_V
(1+\varepsilon_5)E\bigl(m^{N,(k+1)|k}_{t}(\cdot|\x^{[k]})\big|m_t\bigr) \\
&\mathrel{\hphantom{\leqslant}} \quad \negmedspace {} +
M_V (1+\varepsilon_5^{-1})\lVert\nabla\log m_t\rVert_{L^\infty}^2
D\bigl(m^{N,(k+1)|k}_{t}(\cdot|\x^{[k]})\big|m_t\bigr).
\end{align*}
for $M_V \coloneqq
\sup_{t \in[0,T]}\sup_{x\in\T^d} \int_{\T^d} \lvert V(x-y)\rvert^2 m_t(\dd y)$.
Noticing that the towering property holds for
$\chi^2$ distance and Dirichlet energy:
\begin{align*}
\int_{\T^{kd}} \bigl( h^{N,k}_t \bigr)^2
D\bigl(m^{N,(k+1)|k}_{t,\x^{[k]}}\big|m_t\bigr)
m^{\otimes k}_t(\dd\x^{[k]})
&= D^{k+1}_t - D^k_t, \\
\int_{\T^{kd}} \bigl( h^{N,k}_t \bigr)^2
E\bigl(m^{N,(k+1)|k}_{t,\x^{[k]}}\big|m_t\bigr)
m^{\otimes k}_t(\dd\x^{[k]})
&= \frac{E^{k+1}_t}{k+1},
\end{align*}
we integrate against $m_t^{\otimes k}$ and get
\begin{multline*}
B_1 \leqslant \varepsilon_4 E^k_t
+ \frac{(1+\varepsilon_5)M_V(N-k)^2k}
{4\varepsilon_4(N-1)^2(k+1)} E^{k+1}_t \\
+ \frac{(1+\varepsilon_5^{-1})M_V\lVert\nabla\log m_t\rVert_{L^\infty}^2(N-k)^2
k} {4\varepsilon_4(N-1)^2} \bigl( D^{k+1}_t - D^{k}_t \bigr).
\end{multline*}

\subsection{Conclusion of the proof}

By combining the upper bounds on $A_1$, $A_2$, $B_1$, $B_2$
obtained in the previous steps, we get
\begin{align*}
\frac{\dd H^k_t}{\dd t}
&\leqslant - \bigl(1 - {\textstyle\sum_{n=1}^4 \varepsilon_n}\bigr) I^k_t
+ \frac{(1+\varepsilon_5)\lVert V\rVert_{L^\infty}^2}{4\varepsilon_4}
I^{k+1}_t \1_{k < N} \\
&\mathrel{\hphantom{\leqslant}}\quad\negmedspace {}
+ C M_{V,m_t} H^k_t \\
&\mathrel{\hphantom{\leqslant}}\quad\negmedspace {}
+ \biggl(\frac{C\lVert K_2\rVert_{L^\infty}^2}{\varepsilon_3}
+ \frac{(1+\varepsilon_5^{-1})\lVert V\rVert_{L^\infty}^2
\lVert\nabla \log m_t\rVert_{L^\infty}^2} {2\varepsilon_4}\biggr)
k \bigl( H^{k+1}_t - H^k_t\bigr) \1_{k < N} \\
&\mathrel{\hphantom{\leqslant}}\quad\negmedspace {}
+ C M_{V,m_t} \frac{k^2}{N^2}
+ C \biggl( \frac{\lVert K_2\rVert_{L^\infty}^2}{\varepsilon_1}
+ \frac{M_{V,m_t}}{\varepsilon_2} \biggr) \frac{k^2}{N^2} \times
\begin{cases}
k \\
1 + k \sqrt{H^3_t}
\end{cases}
\end{align*}
for the entropic case $p = 1$, and
\begin{align*}
\frac{1}{2}\frac{\dd D^k_t}{\dd t}
&\leqslant - \bigl(1 - {\textstyle\sum_{n=1}^4 \varepsilon_n}\bigr) E^k_t
+ \frac{(1+\varepsilon_5)M_V}{4\varepsilon_4}
E^{k+1}_t \1_{k < N} \\
&\mathrel{\hphantom{\leqslant}}\quad\negmedspace {}
+ C \biggl[
M_{V,m_t} \biggl(1 + \frac{k^2}{N} + \frac{k^3}{\varepsilon_2N^2} \biggr)
+ \frac{\lVert K_2\rVert_{L^\infty}^2k^3}{N^2} \biggr] D^k_t \\
&\mathrel{\hphantom{\leqslant}}\quad\negmedspace {}
+ \biggl(\frac{C\lVert K_2\rVert_{L^\infty}^2}{\varepsilon_3}
+ \frac{(1+\varepsilon_5^{-1})M_V\lVert\nabla\log m_t\rVert_{L^\infty}^2}
{4\varepsilon_4}\biggr)
k \bigl( D^{k+1}_t - D^k_t\bigr) \1_{k < N} \\
&\mathrel{\hphantom{\leqslant}}\quad\negmedspace {}
+ C \biggl( \frac{\lVert K_2\rVert_{L^\infty}^2}{\varepsilon_1}
+ M_{V,m_t} (1+\varepsilon_2^{-1}) \biggr) \frac{k^2}{N^2}
\end{align*}
for the $L^2$ case $p=2$.

Since $\lVert V\rVert_{L^\infty}^2$, $M_V$ are respectively supposed to be
smaller than $1$ in Theorems~\ref{thm:entropy} and \ref{thm:l2},
we can take
\[
\varepsilon_4 = \begin{cases}
\lVert V\rVert_{L^\infty} / 2 & \text{when $p = 1$,} \\
\sqrt{M_V} / 2 & \text{when $p = 2$.}
\end{cases}
\]
so that for $\varepsilon_1$, $\varepsilon_2$, $\varepsilon_3$, $\varepsilon_5$
small enough, we have
\[
1 - {\sum_{n=1}^4 \varepsilon_n}
> \frac{(1+\varepsilon_5)}{4\varepsilon_4} \cdot \begin{cases}
\lVert V\rVert_{L^\infty}^2 & \text{when $p=1$,}\\
M_V & \text{when $p=2$.} \end{cases}
\]
Additionally, for the second assertion of Theorem~\ref{thm:entropy},
since we have
\[
\frac{r_*}{8\pi^2(1 - \lVert V\rVert_{L^\infty})} \leqslant 1,
\]
we can pick the $\varepsilon_i$, for $i \in [3]$ and $i = 5$, such that
\[
1 - {\sum_{n=1}^4 \varepsilon_n}
- \frac{(1+\varepsilon_5)}{4\varepsilon_4} \lVert V\rVert_{L^\infty}^2
= 1 - \frac{2+\varepsilon_5}{2}\lVert V\rVert_{L^\infty}
- {\sum_{i=1}^3}\varepsilon_i
\geqslant \frac{r_*}{8\pi^2}.
\]
Fix these choices of $\varepsilon_i$ for $i \in [5]$
in the respective situations.

Then, for the first assertion of Theorem~\ref{thm:entropy},
we choose the first alternative
in the upper bound of $\dd H^k_t / \!\dd t$, and get
\[
\frac{\dd H^k_t}{\dd t}
\leqslant -c_1 I^k_t + c_2 I^{k+1}_t \1_{k < N}
+ M'_1 H^k_t + M'_2 k \bigl( H^{k+1}_t - H^k_t \bigr) \1_{k < N}
+ M'_3 \frac{k^3}{N^2},
\]
for $c_1 > c_2 \geqslant 0$ and some set of constants $M'_i$, $i\in [3]$.
Applying the first case of Proposition~\ref{prop:ent-hier}
in Section~\ref{sec:hier}
to the system of differential inequalities of $H^k_t$, $I^k_t$,
we get an $M'$ such that
$H^k_t \leqslant M' e^{M't} k^3 \!/ N^2$.
So taking $k=3$, we get the bound on the $3$-marginal's relative entropy:
$H^3_t \leqslant 27M'e^{M't}\!/N^2$.
Plugging this bound into the second alternative
in the upper bound of $\dd H^k_t / \!\dd t$, we get
\[
\frac{\dd H^k_t}{\dd t}
\leqslant -c_1 I^k_t + c_2 I^{k+1}_t \1_{k < N}
+ M_1 H^k_t + M_2 k \bigl( H^{k+1}_t - H^k_t \bigr) \1_{k < N}
+ M_3e^{M_3t} \frac{k^2}{N^2},
\]
for some other set of constants $M_i$, $i \in [3]$.
We apply again the first case of Proposition~\ref{prop:ent-hier} to obtain
the desired result $H^k_t \leqslant M e^{Mt}k^2\!/N^2$.

For the second assertion of Theorem~\ref{thm:entropy},
we have $K_2 = 0$ and
\[
\lVert \nabla \log m_t \rVert_{L^\infty}^2
+ \lVert \nabla^2 \log m_t \rVert_{L^\infty}
\leqslant M_m e^{-\eta t}.
\]
Taking the first alternative in the upper bound
of $\dd H^k_t/\!\dd t$, we get
\begin{multline*}
\frac{\dd H^k_t}{\dd t}
\leqslant -c_1 I^k_t + c_2 I^{k+1}_t \1_{k < N} \\
+ C M_m e^{-\eta t} H^k_t
+ C(1+\varepsilon_5^{-1})
M_m e^{-\eta t}k \bigl( H^{k+1}_t - H^k_t \bigr) \1_{k < N} \\
+ C(1+\varepsilon_2^{-1}) M_m e^{-\eta t}\frac{k^3}{N^2}.
\end{multline*}
Notice that by our choice of constants, we have
\[
c_1 - c_2 \geqslant \frac{r_*}{8\pi^2}.
\]
On the other hand, according to \cite[Proposition 5.7.5]{BGLMarkov},
the uniform measure $1$
on $\T = \R / \mathbb Z$ verifies a log-Sobolev inequality:
\[
\forall m \in \mathcal P (\T)~\text{regular enough},\qquad
8\pi^2 H(m | 1)
\leqslant I(m | 1),
\]
and the inequality with the same $8\pi^2$ constant
for the uniform measure on $\T^d$ by tensorization property.
By the gradient bound $\lVert \nabla \log m_t \rVert_{L^\infty}^2
\leqslant M_m e^{-\eta t}$, we can control the oscillation of $\log m_t$:
\[
\sup_{\T^d} \log m_t - \inf_{\T^d} \log m_t
\leqslant \frac{M_m\sqrt d}{2} e^{-\eta t}.
\]
Thus, by Holley--Stroock's perturbation result
\cite{HolleyStroockLSI},
the measure $m_t$ satisfies
a log-Sobolev inequality with constant
\[
8\pi^2 \exp \biggl( - \frac{M_m \sqrt d}{2} e^{-\eta t} \biggr),
\]
which implies
\[
I^k_t \geqslant \frac{r_*}{c_1 - c_2} H^k_t,
\]
for sufficiently large $t$.
Let $r \in (0, r_*)$ be arbitrary.
We can apply the second case of Proposition~\ref{prop:ent-hier}
and get
\[
H^k_t \leqslant M'' e^{-rt} \frac{k^3}{N^2}.
\]
We then plug the bound for $H^3_t$ back to the second alternative
in the upper bound for $\dd H^k_t/\!\dd t$ to get
\begin{multline*}
\frac{\dd H^k_t}{\dd t}
\leqslant -c_1 I^k_t + c_2 I^{k+1}_t \1_{k < N} \\
+ C M_m e^{-\eta t} H^k_t
+ C(1+\varepsilon_5^{-1})
M_m e^{-\eta t}k \bigl( H^{k+1}_t - H^k_t \bigr) \1_{k < N} \\
+ C(1+\varepsilon_2^{-1}) M_m (1+M'') e^{-\eta t}\frac{k^2}{N^2}.
\end{multline*}
Applying again the second case of Proposition~\ref{prop:ent-hier},
we obtain the desired control
\[
H^k_t \leqslant M' e^{-rt} \frac{k^2}{N^2}.
\]

Finally, in the $L^2$ case, we apply the crude bounds
$k^2\! / N \leqslant k$, $k^3\! / N^2 \leqslant k$,
$D^k_t \leqslant D^{k+1}_t$
in the second line of the upper bound for $\dd D^k_t/\!\dd t$,
and $k\bigl(D^{k+1}_t - D^k_t\bigr) \leqslant kD^{k+1}_t$
in the third line.
So we get
\[
\frac{\dd D^k_t}{\dd t}
\leqslant - c_1 E^k_t + c_2 E^{k+1}_t \1_{k < N}
+ M_2 k D^{k+1}_t \1_{k < N} + M_3 \frac{k^2}{N^2}
\]
for some $c_1 > c_2 \geqslant 0$ and $M_2$, $M_3 \geqslant 0$.
We conclude the proof by
applying Proposition~\ref{prop:l2-hier} in Section~\ref{sec:hier}
to the system of $D^k_t$, $E^k_t$.
\qed

\section{ODE hierarchies}
\label{sec:hier}

\subsection{Entropic hierarchy}

Now we move on to solving the ODE hierarchy
that is ``weaker'' than that considered in \cite{LackerHier}.
As we have seen in the previous section,
in the time-derivative of the $k$-th level entropy $\dd H^k_t/\!\dd t$,
we allow the Fisher information of the next level, i.e.\ $I^{k+1}_t$, to appear.
In this section, we show that
as long as the extra term's coefficient is controlled by the heat dissipation,
the hierarchy still preserves the $O(k^2/N^2)$ order globally in time.
This is achieved by choosing a weighted mix of entropies
at all levels $\geqslant k$ so that
when we consider its time-evolution, a telescoping sequence appears
and cancels all the Fisher informations.

\begin{prop}
\label{prop:ent-hier}
Let $T \in (0, \infty]$ and let $x^k_\cdot$, $y^k_\cdot
: [0,T) \to \R_{\geqslant 0}$ be $\mathcal C^1$ functions, for $k \in [N]$.
Suppose that $x^{k+1}_t \geqslant x^k_t$ for all $k \in [N-1]$.
Suppose that there exist integer $\beta \geqslant 2$,
real numbers $c_1 > c_2 \geqslant 0$ and $C_0 \geqslant 0$,
and functions $M_1$, $M_2$, $M_3 : [0,T) \to [0, \infty)$ such that
for all $t \in [0,T)$ and $k \in [N]$, we have
\begin{align}
x^k_0 &\leqslant \frac{C_0 k^2}{N^2}, \nonumber \\
\frac{\dd x^k_t}{\dd t} &\leqslant - c_1y^k_t + c_2y^{k+1}_t \1_{k < N}
+ M_1(t) x^k_t
+ M_2(t)k\bigl( x^{k+1}_t - x^k_t \bigr) \1_{k < N}
+ M_3(t)\frac{k^\beta}{N^2}.
\label{eq:hier}
\end{align}
Then we have the following results.
\begin{enumerate}
\item If $M_1$, $M_2$ are constant functions
and $M_3(t) \leqslant Le^{Lt}$ for some $L \geqslant 0$,
then there exists $M>0$,
depending only on $\beta$, $c_1$, $c_2$, $C_0$, $M_1$, $M_2$ and $L$,
such that for all $t \in [0,T)$, we have
\[
x^k_t \leqslant Me^{Mt} \frac{k^\beta}{N^2}.
\]
\item If $T = \infty$, the functions $M_1$, $M_2$, $M_3$
are non-increasing and satisfy
\[
M_i(t) \leqslant L e^{-\eta t}
\]
for all $t \in [0,\infty)$ and all $i \in [3]$,
for some $L \geqslant 0$, $\eta > 0$
and if $y^k_t \geqslant \rho x^k_t$ for all $t \in [t_*,\infty)$
for some $\rho > 0$ and some $t_* \geqslant 0$,
then for all $r \in \bigl( 0, \rho(c_1 - c_2) \bigr)$,
there exists $M' \geqslant 0$, depending only on
$r$, $\eta$, $\beta$, $c_1$, $c_2$, $C_0$, $L$, $\rho$ and $t_*$, such that
for all $t \in [0,\infty)$, we have
\[
x^k_t \leqslant M' e^{- \min(r,\eta)t} \frac{k^\beta}{N^2}.
\]
\end{enumerate}
\end{prop}

\begin{proof}
We prove the proposition by considering the two cases at the same time.
Notice that the relation
\[
y^k_t \geqslant \rho x^k_t
\]
trivially holds for $\rho = 0$.
We set $t_* = \infty$ in the first case.
Allowing $\rho$ to be a function of time,
we simply set $\rho(\cdot) = 0$ in the first situation
and in the second situation on the interval $[0,t_*]$
for the rest of the proof.
So formally we can write
\[
\rho(t) = \rho \1_{t\geqslant t_*}.
\]
To avoid confusion we will always write $\rho(\cdot)$
for the time-dependent function and $\rho$ for the constant.

\proofstep{Step 1: Reduction to $M_1 = 0$}
We first notice that, by defining the new variables
\[
x'^k_t = x^k_t \exp\biggl( - \int_0^t M_1(s) \dd s\biggr),~
y'^k_t = y^k_t \exp\biggl( - \int_0^t M_1(s) \dd s\biggr),
\]
we can reduce to the case where $M_1 = 0$ upon redefining $M_3$
(and therefore $L$ in the second case, but not $\eta$).
This transform does not change the relations
\[
x^{k+1}_t \geqslant x^k_t,\qquad y^k_t \geqslant \rho x^{k}_t
\]
and the initial values of $x^k$,
so we can suppose $M_1 = 0$ without loss of generality.

\proofstep{Step 2: Reduction to $k \leqslant N/2$}
Second, by taking $k = N$ in the hierarchy \eqref{eq:hier}, we find
\[
\frac{\dd x^N_t}{\dd t}
\leqslant - \rho(t) x^N_t + M_3(t) N^{\beta - 2}
\]
and thus the a priori bound follows:
\begin{equation}
\label{eq:entropy-a-priori-bound}
x^N_t \leqslant
\biggl(C_0 e^{-\int_0^t\rho}
+ \int_0^t e^{-\int_s^t\rho}M_3(s) \dd s \biggr) N^{\beta - 2}
\eqqcolon M^N_t N^{\beta - 2}
\end{equation}
In the second case where $\rho(\cdot)$ is eventually constant:
$\rho(\cdot) = \rho > 0$,
the quantity $M^N_t$
is exponentially decreasing in $t$ with rate $\min(\rho,\eta)$.
By the monotonicity of $k \mapsto x^k_t$, we get that
for all $k > N/2$,
\[
x^k_t \leqslant x^N_t
\leqslant M^N_t N^{\beta - 2}
< 2^\beta M^N_t \frac{k^\beta}{N^2}.
\]
So it only remains to establish
the upper bound of $x^k_t$ for $k \leqslant N/2$.

\proofstep{Step 3: New hierarchy}
Let $\alpha$ be an arbitrary real number $\geqslant \beta+3$.
Recall that in the second case, $r \in \bigl( 0, \rho(c_1 - c_2)\bigr)$
and in the first case we simply set $r = 0$ and adopt the convention
$0 / 0 = 0$.
Let
\[
i_0\coloneqq
\max \biggl( 1,
\inf \biggl\{ i > 0 : \frac{i^\alpha}{(i+1)^\alpha}
\geqslant \frac{c_2 + r/\rho}{c_1}\biggr\} \biggr).
\]
The number $i_0$ is always well defined,
as $\lim_{i \to \infty} i^{\alpha} \!/ (i+1)^\alpha = 1
> (c_2 + r/\rho) \big/ c_1$.
Thus, for any $i \geqslant i_0$, we have
\[
\frac{c_1}{(i+1)^\alpha} \geqslant \frac{c_2}{i^\alpha}
+ \frac{r}{\rho i^{\alpha}}.
\]
Define, for $k \in [N]$ and $t \geqslant 0$,
the following new variable:
\[
z^k_t \coloneqq
\sum_{i = k}^N \frac{x^i_t}{(i - k + i_0)^\alpha}.
\]
By summing up the ODE hierarchy \eqref{eq:hier} (with $M_1 = 0$), we find
\begin{multline}
\label{eq:z-hierarchy-1}
\frac{\dd z^k_t}{\dd t}
\leqslant - \sum_{i=k}^N \frac{c_1y^i_t}{(i-k+i_0)^\alpha}
+ \sum_{i=k}^{N-1} \frac{c_2y^{i+1}_t}{(i-k+i_0)^\alpha} \\
+ \frac{M_3(t)}{N^2} \sum_{i=k}^N \frac{i^\beta}{(i -k+i_0)^\alpha}
+ M_2(t) \sum_{i=k}^{N-1} \frac{i}{(i-k+i_0)^\alpha}
\bigl(x^{i+1}_t-x^i_t\bigr).
\end{multline}
The sum of the first two terms on the right of \eqref{eq:z-hierarchy-1}
satisfies
\begin{align*}
\MoveEqLeft
- \sum_{i=k}^N \frac{c_1y^i_t}{(i-k+i_0)^\alpha}
+ \sum_{i=k}^{N-1} \frac{c_2y^{i+1}_t}{(i-k+i_0)^\alpha} \\
&= - \frac{c_1y^k_t}{i_0^\alpha}
+ \sum_{i=k}^N \biggl( - \frac{c_1}{(i+1-k+i_0)^\alpha}
+ \frac{c_2}{(i-k+i_0)^\alpha}\biggr) y^i_t \\
&\leqslant -
\sum_{i=k}^N\frac{r\rho(t)y^i_t}{\rho(i-k+i_0)^\alpha}
\leqslant - \sum_{i=k}^N
\frac{rx^i_t}{(i-k+i_0)^\alpha}
= - r z^k_t \1_{t\geqslant t_*},
\end{align*}
thanks to our choice of $i_0$.
The third term on the right of \eqref{eq:z-hierarchy-1} satisfies
\begin{multline}
\label{eq:hier-third}
\sum_{i=k}^N \frac{i^\beta}{(i-k+i_0)^\alpha}
\leqslant C_\beta\sum_{i=k}^N \frac{(i-k)^\beta + k^\beta}{(i-k+1)^\alpha}
\leqslant C_\beta\sum_{i=1}^{\infty} \frac{(i-1)^\beta}{i^\alpha}
+ C_\beta k^\beta \sum_{i=1}^{\infty} \frac{1}{i^\alpha} \\
\leqslant C_{\alpha,\beta} k^\beta,
\end{multline}
where $C_\beta > 0$ (resp.\ $C_{\alpha,\beta} > 0$)
depends only on $\beta$ (resp.\ $\alpha$ and $\beta$).
In the following, we allow these constants to change from line to line.

For the last term on the right of \eqref{eq:z-hierarchy-1},
we perform the summation by parts:
\begin{align*}
\MoveEqLeft
\sum_{i=k}^{N-1} \frac{i}{(i-k+i_0)^\alpha}
\bigl(x^{i+1}_t-x^i_t\bigr) \\
&= - \frac{k}{i_0^\alpha} x^k_t + \frac{N}{(N-k+i_0)^\alpha}x^N_t
+ \sum_{i=k}^{N-1} \biggl( \frac{i}{(i-k+i_0)^\alpha}
- \frac{(i+1)}{(i+1-k+i_0)^\alpha}\biggr) x^{i+1}_t.
\end{align*}
The coefficient in the last summation satisfies
\begin{align*}
\MoveEqLeft \frac{i}{(i-k+i_0)^\alpha}
- \frac{(i+1)}{(i+1-k+i_0)^\alpha} \\
&= \biggl( \frac{1}{(i-k+i_0)^{\alpha-1}}
- \frac{1}{(i+1-k+i_0)^{\alpha-1}} \biggr) \\
&\mathrel{\hphantom{=}}\quad\negmedspace{} + (k-i_0)
\biggl( \frac 1{(i-k+i_0)^\alpha}
- \frac 1{(i+1-k+i_0)^\alpha}\biggr) \\
&\leqslant \frac{\alpha -1}{(i-k+i_0)^\alpha}
+ k\biggl( \frac 1{(i-k+i_0)^\alpha} - \frac 1{(i+1-k+i_0)^\alpha}\biggr),
\end{align*}
where the last inequality is due to
$j^{-\alpha+1} - (j+1)^{-\alpha+1} \leqslant (\alpha-1) j^{-\alpha}$
for $\alpha > 1$ and $j > 0$.
Thus, we have
\begin{align*}
\MoveEqLeft
\sum_{i=k}^{N-1} \frac{i}{(i-k+i_0)^\alpha} \bigl(x^{i+1}_t-x^i_t\bigr) \\
&\leqslant
- \frac{k}{i_0^\alpha} x^k_t
+ \frac{N}{(N-k+i_0)^\alpha} x^N_t
+ (\alpha -1)\sum_{i=k}^{N-1} \frac{x^{i+1}_t}{(i-k+i_0)^\alpha} \\
&\mathrel{\hphantom{\leqslant}}\quad\negmedspace {}
+ k \sum_{i=k}^{N-1} \biggl( \frac 1{(i-k+i_0)^\alpha}
- \frac 1{(i+1-k+i_0)^\alpha} \biggr) x^{i+1}_t
\end{align*}
The difference between $z^{k+1}_t$ and $z^k_t$ reads
\[
z^{k+1}_t - z^k_t
= \sum_{i=k}^{N-1} \biggl( \frac{1}{(i-k+i_0)^\alpha}
- \frac{1}{(i+1-k+i_0)^\alpha} \biggr) x^{i+1}_t
- \frac{x^k_t}{i_0^\alpha}.
\]
Then, rewriting in terms of $z^k_t$ and $z^{k+1}_t$,
we find that, for $k \in [N-1]$, the last summation satisfies
\begin{align*}
\MoveEqLeft \sum_{i=k}^{N-1} \frac{i}{(i-k+i_0)^\alpha}
\bigl(x^{i+1}_t-x^i_t\bigr) \\
&\leqslant \sum_{i=k}^{N-1} \frac{\alpha-1}{(i-k+i_0)^\alpha} x^{i+1}_t
+ k \bigl( z^{k+1}_t - z^k_t \bigr)
+ \frac{N}{(N-k+i_0)^\alpha}x^N_t \\
&\leqslant \frac{(\alpha - 1)c_1}{c_2}
\sum_{i=k+1}^N \frac{x^{i}_t}{(i-k+i_0)^\alpha}
+ k \bigl(z^{k+1}_t - z^k_t\bigr) + \frac{N}{(N-k+i_0)^\alpha}x^N_t\\
&= \frac{(\alpha-1)c_1}{c_2} z^k_t
+ k \bigl(z^{k+1}_t - z^k_t\bigr)
+ \frac{N}{(N-k+i_0)^\alpha}x^N_t.
\end{align*}
Then for $k \leqslant N/2$, we have
\begin{align*}
\MoveEqLeft \sum_{i=k}^{N-1} \frac{i}{(i-k+i_0)^\alpha}
\bigl(x^{i+1}_t-x^i_t\bigr) \\
&\leqslant \frac{(\alpha-1)c_1}{c_2} z^k_t
+ k \bigl(z^{k+1}_t - z^k_t\bigr)
+ \frac{N}{(N/2)^\alpha} x^N_t \\
&\leqslant \frac{(\alpha-1)c_1}{c_2} z^k_t
+ k \bigl(z^{k+1}_t - z^k_t\bigr)
+ \frac{2^\alpha}{N^{\alpha-1}}
M^N_tN^{\beta-2} \\
&\leqslant \frac{(\alpha-1)c_1}{c_2} z^k_t
+ k \bigl(z^{k+1}_t - z^k_t\bigr)
+ \frac{2^\alpha M^N_t}{N^2},
\end{align*}
where the last inequality is due to $\alpha \geqslant \beta + 3$.

Combining the upper bounds for all the terms on the right of
\eqref{eq:z-hierarchy-1}, we get,
for $k \leqslant N/2$,
\begin{multline}
\frac{\dd z^k_t}{\dd t}
\leqslant - r z^k_t \1_{t \geqslant t_*}
+ \frac{(\alpha-1)c_1M_2(t)}{c_2}z^k_t
+ M_2(t) k \bigl(z^{k+1}_t - z^k_t\bigr) \\
+ C_{\alpha,\beta} M_3(t) \frac{k^\beta}{N^2}
+ \frac{2^\alpha M^N_t M_2(t)}{N^2},
\label{eq:new-hier}
\end{multline}
For $k=\bar k\coloneqq\lfloor N/2\rfloor + 1$, we have
by the a priori bound \eqref{eq:entropy-a-priori-bound},
\[
z^{\bar k}_t = \sum_{i=\bar k}^N \frac{x^i_t}{(i-\bar k+i_0)^\alpha}
\leqslant x^N_t \sum_{i=k}^N \frac{1}{(i-\bar k+i_0)^\alpha}
\leqslant C_\alpha M^N_t N^{\beta -2}.
\]
According to the computations in \eqref{eq:hier-third},
the initial values of $z^k_0$, for $k \leqslant N/2$, satisfy
\[
z^k_0 \leqslant C_{\alpha} C_0 \frac{k^2}{N^2}
\eqqcolon C'_0 \frac{k^2}{N^2}.
\]
So the new hierarchy in terms of $z^k_t$ is derived.
\medskip

At this point, we can already apply the Grönwall iteration method
of Lacker \cite{LackerHier}
and, in the time-uniform case, of Lacker and Le Flem \cite{LLFSharp},
to solve the system of differential inequalities \eqref{eq:new-hier}.
However, we take a much simpler approach here
based on the following observation.
If the variable $k$ in \eqref{eq:new-hier} is no longer discrete but continuous,
then the term $M_2(t) k \bigl( z^{k+1}_t - z^k_t \bigr)$ becomes
the transport term
\[
M_2(t) k \frac{\partial z^{k+1}_t}{\partial k},
\]
and $z^k_t$ becomes a subsolution to a transport equation
\[
\frac{\partial z^k_t}{\partial t}
\leqslant - r z^k_t \1_{t \geqslant t_*}
+ M_2(t) k \frac{\partial z^k_t}{\partial k}
+ \text{source terms}.
\]
Since the transport equation verifies a comparison principle,
it suffices to construct a supersolution to the equation
that dominates $z^k_t$ on the parabolic boundary,
in order to obtain an upper bound for $z^k_t$ in the continuous case.
The crucial observation here,
which we prove in Proposition~\ref{prop:maximum-principle} in
Section~\ref{sec:maximum-principle},
is that the comparison still holds
for the discretization scheme \eqref{eq:new-hier}.
So in the following we construct supersolutions
for the system of differential inequalities
in the two cases of the proposition.

\proofstep{Step 4.1: Global-in-time estimates}
In the first case, we can control $M^N_t$
defined in \eqref{eq:entropy-a-priori-bound} by
\[
M^N_t \leqslant C_0 + e^{Lt} - 1.
\]
Thus, by the last step,
\[
z^{\bar k}_t \leqslant C_\alpha (C_0 + e^{Lt} - 1) N^{\beta-2}.
\]
where $\bar k = \lfloor N/2 \rfloor +1$ as we recall.
Now we set, for $k \leqslant N/2$,
\[
w^k_t = M e^{Mt} \frac{k^\beta}{N^2}
\]
for some $M$ to be determined.
For $M$ large enough, we have the domination
\[
w^k_t \geqslant z^k_t
\]
on the parabolic boundary
\[
\{ (t,k) \in [0,\infty) \times [N] :
\text{$t=0$ or $k = \bar k$} \}.
\]
In the interior, $w^k_t$ is an upper solution for \eqref{eq:new-hier}
if and only if
\begin{multline*}
M^2 e^{Mt} \frac{k^\beta}{N^2}
\geqslant \frac{(\alpha-1)c_1M_2}{c_2} Me^{Mt} \frac{k^\beta}{N^2}
+ M_2 \frac{k\bigl( (k+1)^\beta - k^\beta\bigr)}{N^2}
+ C_{\alpha,\beta} \frac{k^\beta}{N^2} \\
+ 2^\alpha M_2 \frac{C_0+e^{Lt}-1}{N^2}.
\end{multline*}
Noting that $(k+1)^\beta - k^\beta \leqslant \beta (k+1)^{\beta-1}
\leqslant 2^{\beta-1}\beta k^{\beta-1}$,
we can let the inequality hold by taking an $M$ large enough.
We conclude in this case by
applying the comparison principle of Proposition~\ref{prop:maximum-principle}
to $w^k_t - z^k_t$.

\proofstep{Step 4.2: Exponentially decaying estimate}
In this case, the a priori bound $M^N_t$ verifies, for some $M''>0$,
\[
M^N_t \leqslant M'' e^{-\min(r,\eta)t}.
\]
We set, for $k \leqslant N/2$,
\[
w^k_t = M'(t) \frac{k^\beta}{N^2}
\]
for some $M' : [0,\infty) \to [0,\infty)$ to be determined.
The domination $w^k_t \geqslant z^k_t$ on the boundary is satisfied if
\begin{align*}
M'(0) &\geqslant C'_0 \\
M'(t) &\geqslant C_\alpha M^N_t,
\end{align*}
In the interior, $w^k_t$ is an upper solution for \eqref{eq:new-hier}
if and only if
\begin{multline*}
\frac{\dd M'(t)}{\dd t} \geqslant
- r \1_{t \geqslant t_*} M'(t)
+ \frac{(\alpha-1)c_1M_2(t)}{c_2} M'(t)
+ M_2(t) \frac{k\bigl((k+1)^{\beta}-k^\beta\bigr)}{k^\beta} M'(t) \\
+ C_{\alpha,\beta} M_3(t) + \frac{2^\alpha M^N_t M_2(t)}{k^{\beta}}.
\end{multline*}
Note that the source terms on the second line
can be bounded by $L'' e^{-\eta t}$ for some $L'' > 0$.
Set
\[
\rho'(t) = r \1_{t \geqslant t_*}
- \biggl(\frac{(\alpha-1)c_1}{c_2} + 2^{\beta-1}\beta\biggr) M_2(t)
\]
and
\[
M'(t) = M'_0 e^{ - \int_0^t \rho'}
+ \int_0^t e^{ -\int_s^t \rho'} L'' e^{-\eta s} \dd s.
\]
We find that all conditions are satisfied for an $M'_0$ sufficiently large.
We fix such $M'_0$ and apply again Proposition~\ref{prop:maximum-principle}
to $w^k_t - z^k_t$ to conclude.
\end{proof}

\subsection{\boldmath$L^2$ hierarchy}

For the ODE system obtained from the $L^2$ hierarchy,
we only show that the $O(1/N^2)$-order bound holds until some finite time.
We note that similar hierarchies have appeared recently in
\cite{BJSNewApproach,BDJDuality}.

\begin{prop}
\label{prop:l2-hier}
Let $T > 0$ and let $x^k_\cdot$, $y^k_\cdot
: [0,T] \to \R_{\geqslant 0}$ be $\mathcal C^1$ functions, for $k \in [N]$.
Suppose that there exist real numbers $c_1 > c_2 \geqslant 0$,
and $C_0$, $M_2$, $M_3 \geqslant 0$ such that
for all $t \in [0,T]$, $k \in [N]$ and $r \in [0,1)$,
\begin{align*}
\sum_{k=1}^N r^k x^k_0 &\leqslant \frac{C_0}{N^2(1-r)^3}, \\
\frac{\dd x^k_t}{\dd t} &\leqslant - c_1y^k_t + c_2y^{k+1}_t \1_{k < N}
+ M_2kx^{k+1}_t \1_{k < N}
+ M_3\frac{k^2}{N^2}.
\end{align*}
Then, there exist $T_*$, $M>0$,
depending only on $\beta$, $c_1$, $c_2$, $C_0$, $M_2$, $M_3$,
such that for all $t \in [0,T_*\wedge T)$, we have
\[
x^k_t \leqslant \frac{Me^{Mk}}{(T_* - t)^3N^2}.
\]
\end{prop}

\begin{proof}
For $t \in [0,T]$ and $r \in [c_2/c_1, 1]$, we define the generating function
(or the Laplace transform) associated to $x^k_t$:
\[
F(t,r) = \sum_{k=1}^{N} r^k x^k_t.
\]
Then, taking the time-derivative of $F(t,r)$, we get
\begin{align*}
\frac{\partial F(t,r)}{\partial t}
&\leqslant - c_1 \sum_{k=1}^N r^k y^k_t
+ c_2 \sum_{k=1}^{N-1} r^k y^{k+1}_t
+ M_2 \sum_{k=1}^{N-1} kr^k x^{k+1}_t
+ \frac{M_3}{N^2} \sum_{k=1}^{N} k^2r^k \\
&\leqslant  - c_1 r y^1_t
+ \sum_{k=2}^N (c_2 - c_1r) r^{k-1} y^{k+1}_t
+ M_2 \sum_{k=1}^{N-1} kr^k x^{k+1}_t
+ \frac{M_3}{N^2} \sum_{k=1}^{N} k^2r^k \\
&\leqslant
M_2 \sum_{k=1}^{N-1} kr^k x^{k+1}_t
+ \frac{M_3}{N^2} \sum_{k=1}^{N} k^2r^k.
\end{align*}
Notice that, by taking partial derivatives in $r$, we get
\begin{align*}
\frac{\partial F(t,r)}{\partial r}
&= \sum_{k=0}^{N-1} (k+1) r^k x^{k+1}_t, \\
\frac{\partial^2}{\partial r^2} \biggl( \frac{1}{1-r} \biggr)
&= \sum_{k=0}^\infty (k+2)(k+1)r^{k}.
\end{align*}
Thus, we find
\[
\frac{\partial F(t,r)}{\partial t}
\leqslant M_2 \frac{\partial F(t,r)}{\partial r}
+ \frac{2M_3}{N^2(1-r)^3}.
\]
The initial condition of $F$ satisfies
\[
F(0,r) = \sum_{k=1}^N r^k x^k_0
\leqslant \frac{C_0}{N^2(1-r)^3}.
\]
Let
\[
T_* = \frac 1{M_2}\biggl( 1 - \frac{c_2}{c_1} \biggr)
\]
and for $t < T_* \wedge T$,
let $(r_s)_{s \in [0,t]}$ be the characteristic line:
\[
r_s = \frac{c_2}{c_1} + M_2 (t-s).
\]
We then have $r_0 \leqslant c_2/c_1 + M_2t$.
Integrating along this line, we get
\begin{align*}
F(t,r_t) &\leqslant F(0, r_0)
+ \frac{2M_3}{N^2} \int_0^t \frac{\dd s}{(1-r_s)^3} \\
&\leqslant
\frac{C_0}{N^2(1-r_0)^3} +
\frac{2M_3}{M_2N^2} \int_{r_t}^{r_0} \frac{\dd r}{(1-r)^3} \\
&\leqslant \biggl( \frac{C_0}{(1-r_0)^3}
+ \frac{M_3}{M_2(1-r_0)^2} \biggr) \frac{1}{N^2}.
\end{align*}
Thus we get
\[
x^k_t \leqslant r_t^{-k} F(t,r_t)
\leqslant \biggl( \frac{c_1}{c_2} \biggr)^{\!k}
\biggl( \frac{C_0}{\bigl(1 - M_2 t - \frac{c_2}{c_1}\bigr)^3}
+ \frac{M_3}{M_2\bigl(1 - M_2 t - \frac{c_2}{c_1}\bigr)^2} \biggr)
\frac{1}{N^2}. \qedhere
\]
\end{proof}

\begin{rem}
Proposition~\ref{prop:l2-hier} provides only $L^2$ estimates
on finite horizons, in contrast to the global-in-time result
of \textcite{HCRHigher}.
The limitation arises because, for singular interactions,
the hierarchy cannot be forced to stop at a level
$k \sim N^{\alpha}$ with $\alpha \in (0,1)$,
as no sufficiently strong a priori estimates are available
for $x_t^k$ and $y_t^k$.
\end{rem}

\section{Other technical results}
\label{sec:toolbox}

\subsection{Transport inequality for \boldmath$W^{-1,\infty}$ kernels}

One key ingredient of the entropic hierarchy of Lacker \cite{LackerHier}
is to control the outer interaction terms by the relative entropy
through the Pinsker or Talagrand's transport inequality.
In our situation, the interaction force kernel is more singular,
and we are no longer able to control the difference
by the mere relative entropy.
It turns out that the additional quantity to consider is
the relative Fisher information.%
\footnote{It has been communicated to the author
that Lacker has also obtained the inequality independently.}
Similar estimates appear in \cite[Section~2.1]{JabinWang}.
We also include the inequality for the $L^2$ hierarchy here,
as the two inequalities share the same form.

\begin{prop}
\label{prop:transport}
For all $K = \nabla \cdot V$ with $V \in L^\infty(\T^d; \R^d\times\R^d)$
and all regular enough measures $m_1$, $m_2 \in \mathcal P(\T^d)$,
we have
\begin{align*}
\lvert\langle K, m_1 - m_2\rangle\rvert
&\leqslant \lVert V\rVert_{L^\infty}
\Bigl( \sqrt{I(m_1|m_2)}
+ \lVert \nabla\log m_2\rVert_{L^\infty} \sqrt{2H(m_1|m_2)} \Bigr), \\
\lvert\langle K, m_1 - m_2\rangle\rvert
&\leqslant \lVert V\rVert_{L^2(m_2)}
\Bigl( \sqrt{E(m_1|m_2)}
+ \lVert \nabla\log m_2\rVert_{L^\infty} \sqrt{D(m_1|m_2)} \Bigr).
\end{align*}
\end{prop}

\begin{proof}
For the first inequality, we have
\begin{align*}
\MoveEqLeft \lvert \langle K, m_1 - m_2\rangle\rvert \\
&= \lvert \langle V, \nabla m_1 - \nabla m_2\rangle\rvert \\
&\leqslant \int_{\T^d} \lvert V\rvert
\biggl| \frac{\nabla m_1}{m_1} - \frac{\nabla m_2}{m_2}\biggr| \dd m_1
+ \int_{\T^d} \frac{\lvert \nabla m_2\rvert}{m_2}
\lvert V\rvert \dd \lvert m_1 - m_2 \rvert \\
&\leqslant \lVert V\rVert_{L^\infty}
\biggl(\int_{\T^d} \Bigl|\nabla \log \frac{m_1}{m_2}\Bigr|^2
\dd m_1\biggr)^{\!1/2}
+ \lVert \nabla \log m_2 \rVert_{L^\infty} \lVert V\rVert_{L^\infty}
\lVert m_1 - m_2 \rVert_{L^1} \\
&\leqslant \lVert V\rVert_{L^\infty}
\Bigl( \sqrt{I(m_1|m_2)}
+ \lVert \nabla\log m_2\rVert_{L^\infty} \sqrt{2H(m_1|m_2)} \Bigr).
\end{align*}
For the second inequality, we set $h = m_1 / m_2$ and find
\begin{align*}
\MoveEqLeft\lvert \langle K_1, m_1 - m_2 \rangle \rvert \\
&= \biggl|\int_{\T^d} K (h - 1) \dd m_2 \biggr| \\
&\leqslant \biggl|\int_{\T^d} V\nabla h \dd m_2\biggr|
+ \biggl|\int_{\T^d} V (h - 1) \nabla \log m_2 \dd m_2 \biggr|\\
&\leqslant
\lVert V\rVert_{L^2(m_2)}
\bigl( \lVert \nabla h \rVert_{L^2(m_2)}
+ \lVert \nabla \log m_t\rVert_{L^\infty}
\lVert h-1 \rVert_{L^2(m_2)} \bigr). \qedhere
\end{align*}
\end{proof}

\subsection{Improved Jabin--Wang lemma}
\label{sec:jw-lem}

We state here a refinement of \cite[Theorem~4]{JabinWang},
which yields the correct asymptotic behavior
when the ``test function'' (denoted there by $\phi$)
tends to zero. This refinement is not needed
for the global analysis in \cite{JabinWang},
but it is essential for bounding the inner interaction
in our local analysis.
For simplicity, we restrict to bounded $\phi$,
which already suffices in the torus setting.
Under the exponential moment condition of \cite{JabinWang},
similar estimates would follow.

We denote the universal constant from \cite{JabinWang} by
\[
\CJW \coloneqq 1600^2 + 36e^4.
\]
A sharper constant is available in
\textcite[Lemma~4.3]{LLNQuantitative}.

\begin{thm}[Alternative version of {\cite[Theorem 4]{JabinWang}}]
\label{thm:concentration}
Let $\phi \in L^\infty( \T^d \times \T^d; \R)$ and $m \in \mathcal P(\T^d)$
be such that $\int_{\T^d} \phi(x,y) m(\dd y)
= \int_{\T^d} \phi(y,x) m(\dd y) = 0$
and $\phi(x,x) = 0$ for all $x\in\T^d$.
Denote $\gamma = \CJW \lVert \phi \rVert_{L^\infty}^2$.
If $\gamma \in \bigl[0,\frac 12\bigr]$, then for all integer $k \geqslant 1$,
we have
\[
\log \int_{\T^{kd}} \exp \biggl( \frac{1}{k} \sum_{i,j \in [k]}
\phi(x^i, x^j) \biggr) m^{\otimes k}(\dd\x^{[k]}) \leqslant
6\gamma.
\]
\end{thm}

The proof will depend on two combinatorical estimates in \cite{JabinWang},
which we state here for the readers' convenience.

\begin{prop}[{\cite[Propositions 4 and 5]{JabinWang}}]
\label{prop:counting}
Under the assumptions of Theorem~\ref{thm:concentration},
for all integer $r \geqslant 1$, we have
\[
\frac 1{(2r)!} \int_{\T^{kd}} \biggl| \frac 1k \sum_{i,j\in[k]}
\phi(x^i, x^j)\biggr|^{2r} m^{\otimes k}(\dd\x^{[k]})
\leqslant \begin{cases}
(6e^2 \lVert \phi \rVert_{L^\infty})^{2r} & \text{if $4r > k$,} \\
(1600 \lVert \phi \rVert_{L^\infty})^{2r} &
\text{if $4 \leqslant 4r \leqslant k$,}
\end{cases}
\]
\end{prop}

\begin{proof}[Proof of Theorem~\ref{thm:concentration}]
Let $a \neq 0$. We have the elementary inequality
\begin{align*}
e^a - a - 1 &= \sum_{r=2}^{\infty} \frac{a^r}{r!}
\leqslant \sum_{r=2}^{\infty} \frac{\lvert a\rvert^r}{r!} \\
&= \sum_{r=1}^{\infty} \frac{\lvert a\rvert^{2r}}{(2r)!}
+ \sum_{r=1}^{\infty} \frac{\lvert a\rvert^{2r+1}}{(2r+1)!} \\
&\leqslant \sum_{r=1}^{\infty} \frac{\lvert a\rvert^{2r}}{(2r)!}
+ \sum_{r=1}^{\infty} \frac{\lvert a\rvert^{2r+1}}{2(2r+1)!}
\biggl( \frac{\lvert a\rvert}{2r+2} + \frac{2r+2}{\lvert a\rvert} \biggr) \\
&\leqslant 3 \sum_{r=1}^\infty \frac{\lvert a\rvert^{2r}}{(2r)!}.
\end{align*}
The inequality $e^a - a - 1 \leqslant 3 \sum_{r=1}^{\infty}
\frac{\lvert a\rvert^{2r}}{(2r)!}$ holds true for $a = 0$ as well.
Taking $a = \frac 1k \sum_{i,j\in[k]} \phi(x^i,x^j)$
in the inequality above and integrating with $m^{\otimes k}(\dd\x^{[k]})$,
we get
\begin{align*}
\MoveEqLeft \int_{\T^{kd}} \exp \biggl( \frac{1}{k} \sum_{i,j \in [k]}
\phi(x^i, x^j) \biggr) m^{\otimes k}(\dd\x^{[k]}) \\
&\leqslant 1 + \frac 1k\sum_{i,j\in[k]}
\int_{\T^{kd}} \phi(x^i,x^j) m^{\otimes k}(\dd\x^{[k]}) \\
&\mathrel{\hphantom{\leqslant}}\hphantom{1}
+ 3 \sum_{r=1}^\infty \frac{1}{(2r)!}
\int_{\T^{kd}} \biggl| \frac 1k \sum_{i,j\in[k]}
\phi(x^i, x^j)\biggr|^{2r} m^{\otimes k}(\dd\x^{[k]}).
\end{align*}
The second term on the right hand side vanishes,
as by assumption, for $i \neq j$,
we have $\int_{\T^{kd}} \phi(x^i,x^j)m^{\otimes k}(\dd\x^{[k]}) = 0$,
and for $i = j$, we have $\phi(x^i,x^i) = 0$.
Thus, using the counting result of Proposition~\ref{prop:counting}, we get
\begin{align*}
\MoveEqLeft \int_{\T^{kd}} \exp \biggl( \frac{1}{k} \sum_{i,j \in [k]}
\phi(x^i, x^j) \biggr) m^{\otimes k}(\dd\x^{[k]}) \\
&\leqslant 1 + 3\sum_{r=1}^{\lfloor k/4\rfloor}
(1600 \lVert \phi\rVert_{L^\infty})^{2r}
+ 3\sum_{r=\lfloor k/4\rfloor+1}^{\infty}
(6e^2\lVert \phi\rVert_{L^\infty})^{2r}
= 1 + \frac{3\gamma}{1 - \gamma}
\end{align*}
We conclude by noting that
$\log\bigl(1 + \frac{3\gamma}{1 - \gamma}\bigr)
\leqslant \frac{3\gamma}{1 - \gamma} \leqslant 6\gamma$
for $\gamma \in \bigl[0, \frac 12\bigr]$.
\end{proof}

Then, taking a rescaling of $\phi$, we get the following.

\begin{cor}
\label{cor:concentration}
Suppose that the function
$\phi \in L^\infty( \T^d \times \T^d; \R)$
and the measure $m \in \mathcal P(\R^d)$
satisfy $\int_{\T^d} \phi(x,y) m(\dd y)
= \int_{\T^d} \phi(y,x) m(\dd y) = 0$
and $\phi(x,x) = 0$ for all $x\in\T^d$.
Then, for all integer $N\geqslant 2$ and $k \in [N]$, we have
\[
\log \int_{\T^{kd}} \exp \biggl( \frac{1}{N} \sum_{i,j \in [k]}
\phi(x^i, x^j) \biggr) m^{\otimes k}(\dd\x^{[k]}) \leqslant
6\CJW\lVert\phi\rVert_{L^\infty}^2 \frac{k^2}{N^2},
\]
given that $\CJW\lVert\phi\rVert_{L^\infty}^2\leqslant 1/2$.
\end{cor}

\subsection{Maximum principle}
\label{sec:maximum-principle}

We state a maximum principle for a system of ODEs. The result can be
proved by a standard contradiction argument, which we omit here and
leave to the reader.

\begin{prop}
\label{prop:maximum-principle}
Let $T > 0$ and
let $\x : [0,T] \to \R^N$ be a $\mathcal C^1$ continuous function.
Suppose that every component of the initial value $\x(0)$
is non-negative, i.e., $x^i(0) \geqslant 0$ for all $i \in [N]$.
Suppose that it satisfies
\[
\forall t \in [0,T],~\forall i \in [N],\qquad
\frac{\dd x^i(t)}{\dd t} \geqslant
\sum_{j \in [N]} A^i_j(t) x^j (t)
\]
for some continuous matrix-valued $A : [0,T] \to \R^{d \times d}$
whose off-diagonal elements are non-negative, i.e.,
$A^i_j (t) \geqslant 0$ for all $i$, $j \in [N]$ such that $i \neq j$.
Then, for all $t \in [0,T]$ and all $i \in [N]$,
we have $x^i(t) \geqslant 0$.
\end{prop}

\sloppy
\printbibliography
\medskip
\myauthorinfo

\end{document}